\documentclass[10pt,twoside]{amsart}
\usepackage{amsmath, amsfonts, amsthm, amssymb, verbatim,a4wide} 
\usepackage[mathcal]{euscript}
\usepackage{graphicx}
\newtheorem{theorem}{Theorem}[section]
\newtheorem{proposition}[theorem]{Proposition}

\newtheorem{definition}[theorem]{Definition}
\newtheorem{remark}[theorem]{\it Remark\/}
\newtheorem{example}[theorem]{\it Example\/}
\newtheorem{lemma}[theorem]{Lemma}
\numberwithin{equation}{section}
\numberwithin{figure}{section}

\newcommand \ee			{\end{equation}}
\newcommand \be		{\begin{equation}}
\newcommand \mathscr {}
\newcommand \RR                 {\mathbb{R}}
\newcommand \RN                 {\mathbb{R}^N}
\newcommand \del                \partial
\newcommand \eps                \epsilon
\newcommand \sig                \sigma

\newcommand \calB                 {\mathcal{B}}
\newcommand \calO                 \Phi 
\newcommand \scrO                 \Psi 
\newcommand \calU                 \Omega 

\newcommand \calW                R 
\newcommand \scrW                R

\newcommand \sgn                {\text{sgn}}

\newcommand \hatl               {\widehat l}
\newcommand \hatr               {\widehat r}
\newcommand \lam                {\lambda}

\newcommand \hatlam             {\widehat\lambda}
\newcommand \bLam               {{\underline \Lambda}}
\newcommand \Lamb               {{\overline \Lambda}}

\newcommand \la                 \langle
\newcommand \ra                 \rangle

\newcommand \ut        {\tilde u}
\newcommand \vt        {\tilde v}
\newcommand{\cplL}{\theta_-}
\newcommand{\cplR}{\theta_+}
\newcommand{\cplLR}{\theta_{\pm}}
\newcommand{\lpcL}{\gamma_-}
\newcommand{\lpcR}{\gamma_+}
\newcommand{\lpcLR}{\gamma_{\pm}}

\begin{document} 
\bibliographystyle{plain}  
\title[Coupling techniques for nonlinear hyperbolic equations]{
Coupling techniques for nonlinear hyperbolic equations. I. Self-similar diffusion for thin interfaces
}
\author[Benjamin Boutin, Fr\'ed\'eric Coquel, and Philippe G. L{\tiny e}Floch]
{Benjamin Boutin$^{1,2}$, Fr\'ed\'eric Coquel$^2$, 
and 
Philippe G. L{\smaller e}Floch$^2$} 
\thanks{
$^1$ Institut de Recherche Math\'ematiques de Rennes (IRMAR), Universit\'e de Rennes 1, 
263 Av. General Leclerc, 35042 Rennes. Email: {\tt benjamin.boutin@univ-rennes1.fr}
\newline
$^2$ Laboratoire Jacques-Louis Lions \& Centre National de la Recherche Scientifique,  
 Universit\'e Pierre et Marie Curie (Paris 6), 4 Place Jussieu, 75252 Paris, France.  
\newline 
Blog:  {\tt philippelefloch.org.}
Email: \texttt{coquel@ann.jussieu.fr, contact@philippelefloch.org} 
}

\date{Published in Proc. A Roy. Soc. Endburgh 141A (2011), 921--956.}
  
\begin{abstract} 
We investigate various analytical and numerical 
techniques for the coupling of nonlinear hyperbolic systems and, in particular, we introduce here 
an augmented formulation which allows for the modeling of the dynamics of 
interfaces between fluid flows. The main technical difficulty 
to be overcome lies in the possible resonance effect when wave speeds coincide and global hyperbolicity is lost. 
As a consequence, non-uniqueness of weak solutions is observed for the initial value problem which need to be supplemented with further admissibility conditions. This first paper is devoted to investigating these issues in the setting of 
self-similar vanishing viscosity approximations to the Riemann problem for general hyperbolic systems. 
Following earlier works by Joseph, LeFloch, and Tzavaras, we establish an existence theorem for the Riemann problem 
under
fairly general structural assumptions on the nonlinear hyperbolic system and its regularization. Our main contribution 
consists of nonlinear wave interaction estimates for solutions which apply to resonant wave patterns.
\end{abstract}
  
\maketitle


\section{Introduction}
\label{section1}

This is the first part of a series devoted to analytical and numerical techniques relevant 
for the coupling of nonlinear hyperbolic systems. We mainly discuss an augmented formulation 
which allows for the modeling of the dynamics of interfaces between fluid flows. The main technical difficulty 
overcome here for the Riemann problem (that is, a Cauchy problem with piecewise constant initial data)  
lies in the possible resonance effect when wave speeds coincide and global hyperbolicity is lost. 
As a consequence, non-uniqueness of weak solutions is observed for the initial value problem which need 
to be supplemented with further admissibility conditions. In the present paper, we restrict attention to self-similar vanishing viscosity approximations to the Riemann problem for general hyperbolic systems.

Specifically, we are interested in the following class of nonlinear hyperbolic systems of $(N+1)$ partial differential equations
\be
\label{eq:UV}
\aligned
A_0(u,v) \, \del_tu + A_1(u,v) \, \del_x u &= 0,
\\
\del_t v &= 0,
\endaligned
\ee
where the vector-valued field $u=u(t,x)\in \RR^N$ and the scalar function $v=v(t,x)\in\RR$ (with $x \in \RR$ and $t \geq 0$)
are the main unknowns of the theory. We assume that the first set of equations in \eqref{eq:UV} forms a strictly hyperbolic
system but admits one wave speed that changes sign, so that the overall system \eqref{eq:UV} is only weakly hyperbolic. 
Specifically, the mappings $A_0,A_1$ are assumed to be smooth, matrix-valued maps
such that $A_0$ is invertible  so that 
the first set of equations in \eqref{eq:UV} is formally equivalent to the following 
nonconservative system with variable coefficients: 
\be
\label{nonc}
\del_t u + A_0(u,v)^{-1} \, A_1(u,v) \, \del_x u = 0. 
\ee
It is assumed that 
the product matrix $A_0(u,v)^{-1} \, A_1(u,v)$ admits real and distinct eigenvalues denoted by $\lambda_i(u,v)$, 
$1 \leq i \leq N$.  
Finally, it is assumed that one eigenvalue $\lambda_m$ may take values about the origin. 
For instance, there might exist a state  $(u^\star,v^\star)\in\RR^N\times\RR$
such that the matrix $A_1(u^\star,v^\star)$ is non-invertible, with 
\be
\label{Resonance}
\lambda_m(u^\star,v^\star) = 0.
\ee
In view of the above assumption, the system \eqref{eq:UV} is called a {\sl weakly hyperbolic system}. 
Our objective, precisely, is to study this {\sl resonant regime.}

Recall that Dafermos~\cite{Dafermos73,Dafermos73b,DafermosDiPerna76} advocated the use
of self-similar regularizations in order to capture the whole wave fan structure of 
weak solutions to the Riemann problem. This consists in 
searching for self-similar solutions depending only on the variable $\xi := x/t$ and, then, 
introducing a self-similar regularization of the given hyperbolic system. Specifically, 
for the problem of coupling under consideration in this paper
we propose, in the
variable $(x,t)$, to regularize \eqref{eq:UV} in the form 
$$ 
\begin{aligned}
   A_0(u^\eps,v^\eps)\del_t u^\eps + A_1(u^\eps,v^\eps)\del_x u^\eps 
   & = \eps t\, \del_x\bigl(B_0(u^\eps,v^\eps)\del_x u^\eps\bigr),
   \\
   \del_t v^\eps &= \eps^pt\,\del_{xx}v^\eps, 
\end{aligned}
$$
where $\eps>0$ is a small parameter and $B_0=B_0(u,v)$ is a given matrix referred to as the viscosity matrix
and $p>0$ is a real parameter.  
In the self-similar variable $\xi$, the equations satisfied by the viscous solutions 
$(u^\eps, v^\eps) = (u^\eps(\xi),v^\eps(\xi))$ 
read (with $\xi\in\RR$)
\be
\label{1.3} 
\begin{aligned}
   \big(-\xi A_0(u^\eps,v^\eps)+ A_1(u^\eps,v^\eps)\big)u^\eps_\xi &= \eps
   \bigr(B_0(u^\eps,v^\eps)u^\eps_\xi\bigr)_\xi,
   \\
   -\xi v^\eps_\xi &= \eps^p v^\eps_{\xi\xi}. 
\end{aligned}
\ee 

Our objective is to study the existence and regularity of solutions to \eqref{1.3} and to rigorously 
justify 
the passage to the limit $\eps \to 0$. Following earlier works by Joseph, LeFloch, and Tzavaras (see references below), 
we are going to establish a uniform ($\eps$-independent) bound on the total variation $TV(u^\eps)$, and 
an existence theorem under fairly general structural assumptions on the hyperbolic system 
and its regularization. 

This general strategy was proposed in the case $B_0 = I$ and $A_0 = I$ by Tzavaras~\cite{Tzavaras96} 
(for conservative systems) and extended by LeFloch and Tzavaras~\cite{LeFlochTzavaras96}
(for non-conservative systems). The technique was further developed 
by Joseph and LeFloch in the series of papers \cite{JosephLeFloch99,JosephLeFloch02b,JosephLeFloch02,JosephLeFloch5,JosephLeFloch6,JosephLeFloch07}. 
For the purpose of the present paper, we will 
especially build on \cite{JosephLeFloch07} where a general technique to derive interaction estimates was introduced and general matrices $B_0$ 
were dealt with. 
For other results on self-similar limits including viscosity-capillarity terms and 
large data, we refer to pioneering works by Slemrod~\cite{FanSlemrod93}, Slemrod and Tzavaras~\cite{SlemrodTzavaras89},
Fan and Slemrod~\cite{FanSlemrod93}.

The coupling that we are studying in the present work may be non-conservative in nature (cf.~Section~\ref{section2} for details). In contrast, for coupling techniques based on systems in conservation form, a large 
literature is available; see for 
instance~\cite{AdimurthiMishraGowda05, AudussePerthame05, BurgerKarlsen08, BurgerKarlsenRisebroTowers04, Diehl96, SeguinVovelle03}.

An outline of this paper follows. In Section~\ref{section2}, we present a general 
approach involving the coupling 
of nonlinear hyperbolic systems. In Section~\ref{section3} we discuss the case of scalar-valued unknowns $u$, and 
establish a general existence theorem for the viscous self-similar Riemann problem \eqref{1.3} when $N=1$. 
As $\eps \to 0$, we prove that this smooth solution converges to an entropy solution, at least in each half-space $x<0$ and $x>0$. 
This global existence result requires no smallness assumption on the data nor on the coupling of the two models.

The core part of this paper is contained in Sections~\ref{section4} and~\ref{construction} which 
cover general systems of $N$ equations. 
Imposing a natural smallness condition on the Riemann data and the coupling of the two models, 
we establish the existence of smooth solutions to the viscous problem, even in the presence of a resonance effect. 
In Section~\ref{section4} we derive the main estimates on the total variation while, in 
Section~\ref{construction}, we justify the limit $\eps \to 0$. We refer to the forthcoming works \cite{BCL2,BCL3,BCL4} 
for further investigation of these solutions and related issues.  


\section{A formulation of the coupling of nonlinear hyperbolic systems}
\label{section2}
 
Before we can state our new formulation based on an augmented system, 
we start by briefly explaining the formulation based on a fixed interface. 
The weakly hyperbolic problem mentioned in the introduction arises, in particular, via the following coupling technique. 
Consider two strictly hyperbolic systems posed in half-spaces: 
\be
\label{HalfProblem}
\aligned
 \del_t w + \del_x f_-(w) = 0,\qquad x<0, \quad t>0,
 \\
 \del_t w + \del_x f_+(w) = 0,\qquad x>0, \quad t>0,
\endaligned
\ee
where the flux $f_\pm$ are given smooth maps defined on open subsets $\Omega_\pm \subset \RR^N$
 and the unknown of the problem is $w=w(t,x) \in \Omega_- \cup \Omega_+$. 
In addition to initial data, a certain
 coupling condition must be prescribed at the (fixed) interface $\{x\!=\!0\}$. 
This problem can be regarded as a boundary and initial value problem within each
half-problem, and the fundamental question is how to formulate a physically relevant boundary
condition so that the global problem is well-posed. One natural requirement, 
following Godlewski and Raviart~\cite{GodlewskiRaviart04,GodlewskiLeThanhRaviart05}, 
is imposing the continuity condition 
\be
\label{Coupling}
\cplL(w(0-,t))=\cplR(w(0+,t)),\quad t>0,
\ee
where $\cplL,\cplR$ are two invertible maps in $\RN$,
with inverses 
$$
\lpcL :=\cplL^{-1}, \qquad \lpcR:=\cplR^{-1}.
$$
These functions precisely provide the necessary freedom to handle various types of 
couplings. For example, by choosing $\cplL=\cplR=\mathrm{Id}$ one imposes the continuity of the variable 
$w$ at the interface, while by choosing
$\cplLR=f_{\pm}$ one imposes the continuity of the flux at the interface 
(so that the general problem is conservative).

Recalling Dubois and LeFloch's theory \cite{DuboisLeFloch88} of the initial and boundary value problem
for nonlinear hyperbolic systems, it is clear that the condition \eqref{Coupling} is 
realistic only when the boundary is {\sl not characteristic,}
 that is when all eigenvalues are bounded away from $0$. 
In the latter case, instead, following the weak formulation of the boundary conditions 
proposed in 
\cite{DuboisLeFloch88} generalized, for the coupling problem, by 
Godlewski and Raviart~\cite{GodlewskiRaviart04,GodlewskiLeThanhRaviart05}
and Ambroso et al.~\cite{GDT06}, we impose 
that the interface condition is satisfied {\sl in a weak sense, only}
and, specifically, takes the form 
\be
\label{CouplingWeak}
\aligned
w(0+,t)\in\calO_+\bigl(\cplR\!\circ\!\cplL^{-1}(w(0-,t))\bigr),\\
w(0-,t)\in\calO_-\bigl(\cplL\!\circ\!\cplR^{-1}(w(0+,t))\bigr).
\endaligned
\ee
where $\calO_+(b_+)$ (resp. $\calO_-(b_-)$) is the Dubois-LeFloch's set of
admissible traces of the associated Riemann solutions 
$$ 
\aligned
\calO_+(b_+) :=\bigl\{\calW_+(0+;b_+,a), \quad a\in\Omega_+ \bigr\},\\
\calO_-(b_-) :=\bigl\{\calW_-(0-;a,b_-), \quad a\in\Omega_- \bigr\}.
\endaligned
$$
Here, $R=\calW_+(x/t;b_+,a)$ denotes the solution of the Riemann problem
$$
\del_t R + \del_x f_+(R)=0,\qquad x\in\RR, \quad t>0,
\qquad
R(x,0)=\begin{cases}
b_+, &x<0,\\
a, &x>0,
\end{cases}
$$
and similarly
$R=\calW_-(x/t; a,b_-)$ is the solution of 
$$
\del_t R + \del_x f_-(R)=0,\qquad x\in\RR,\quad t>0,
\qquad
R(x,0)=\begin{cases}
a, &x<0,\\
b_-, &x>0.
\end{cases}
$$

Yet, when $f_- \neq f_+$, 
the question of the existence and uniqueness of weak solutions satisfying \eqref{CouplingWeak} is a challenging issue. 
In the present work, we propose to {\sl reformulate the above problem} by ``removing'' the
interface and defining a new problem posed on the whole space $\RR$.

We proceed as follows. First of all, we define the {\sl new variables} 
\be
\label{ChangeVariable}
u_- := \cplL(w),\qquad u_+: =\cplR(w), \qquad u:= \begin{cases}
u_-, \qquad & x <0, 
\\
u_+, \qquad & x >0,
\end{cases}
\ee
and we rewrite the half-space problems in the (conservative) form
\be
\label{eq:U}
\del_t 
\lpcLR(u)+\del_x f_{\pm}(\lpcLR(u))=0,\quad
\pm x>0,\quad t>0,
\ee
or equivalently in the (nonconservative) form
\be
\label{eq:U-nonc}
(D_u \gamma_\pm(u)) \, \del_t u + (D_\gamma f_{\pm})(\lpcLR(u)) (D_u \gamma_\pm(u)) \del_x  u = 0,\quad
\pm x>0,\quad t>0.
\ee
The coupling condition becomes 
\be
\label{UCouplingWeak}
\aligned
u(0+,t)\in\scrO_+(u(0-,t)),\\
u(0-,t)\in\scrO_-(u(0+,t)),\\
\endaligned
\ee
where $\scrO_+(b)$ (and similarly $\scrO_-(b)$) is the following set of admissible trace at $\xi=0+$
$$ 
\scrO_+(b)=\bigl\{\scrW_+(0+,b,\mathrm{u}_+),\mathrm{u}_+\in\calU\bigr\},
$$
and $u=\scrW_+(\cdot,b,u_+)$ is the self-similar solution of the following Cauchy problem
$$
\del_t \lpcR(u)+\del_x f_+(\lpcR(u))=0,\quad x\in\RR,t>0,
\qquad
u(x,0)=\begin{cases}
b, &x<0,\\
\mathrm{u}_+, &x>0.
\end{cases}
$$
In absence of a resonance phenomenon, this reformulation allows us to simply impose the 
{\sl continuity of $u$ at interface} 
\be
\label{StateCoupling}
u(0-,t)=u(0+,t).
\ee

Second, we propose to replace the problem \eqref{eq:U}-\eqref{UCouplingWeak} by the new problem
(already mentioned in the introduction) 
\be
\label{eq:UV2}
\begin{aligned}
A_0(u,v)\del_t u + A_1(u,v)\del_x u&= 0,\\
\del_t v &=0,
\end{aligned} 
\ee
which is a nonlinear hyperbolic system in nonconservative form \cite{LeFloch1,LeFloch2}
and where 
$v: [0, +\infty) \times \RR \to [-1,1]$ will be referred to as the {\sl color function.} 
We arrange that regions where $v=-1$ correspond to the left-hand half-problem while
regions where $v=1$ correspond to the right-hand half-problem, by
requiring the following consistency property on $A_0,A_1$: 
\be
\label{Consistancy}
\begin{aligned}
A_0(u,\pm 1)&= D_u\lpcLR(u),\\
A_1(u,\pm 1)&= D_\gamma f_{\pm}\bigl(\lpcLR(u)\bigr) \, D_u\lpcLR(u).
\end{aligned}
\ee
and assuming the existence of a function $C=C(u,v)$ so that 
$$
A_0(u,v)= D_u C(u,v), 
$$ 
and $C(u,\pm 1)= \gamma_\pm(u)$. 
For $j=0,1$, by definition, the matrices $A_j(u,v)$ 
should smoothly connect $A_j(u,-1)$ to $A_j(u,1)$ as $v$ describes the interval $[-1,1]$.
Moreover $A_0$ must be invertible and
$A_0^{-1} \, A_1$ have real and distinct eigenvalues,   
extending here the strict hyperbolicity of the original hyperbolic
half-problems. 

The system \eqref{eq:UV2} is then supplemented with the initial data
\be
\label{Initial}
\begin{aligned}
u_0(x,0)=u_0(x)=:\cplLR(w_0(x))&, \quad \pm x>0\\
v_0(x,0)=v_0(x):=\pm 1&, \quad \pm x>0, 
\end{aligned}
\ee
for some given data $u_0$. 

We are especially interested in the case that the interface is characteristic for some state value 
$(u^\star,v^\star)$, that is, when the matrix $A_1(u^\star,v^\star)$
admits the eigenvalue $0$ and  
\eqref{StateCoupling} need not be satisfied as an equality, in general, so that the weak formulation above is necessary.

 
\section{Existence theory for scalar conservation laws}
\label{section3}

\subsection{Riemann problem with diffusion}
 
In the present section, we restrict attention to scalar equations and 
provide a rather complete study of the problem described in the introduction.  
Note that the problem under consideration is nonconservative in nature, 
and reduces to a conservative system when the component $v$ takes the values $\pm 1$. 
As explained earlier, we search for a function $u= u(\xi)$ obtained as
the limit of smooth approximations $u^\eps, v^\eps$ to   
\be
 \label{eq5}
\aligned
\big(-\xi A_0(u^\eps,v^\eps) + A_1(u^\eps,v^\eps)\big) u^\eps_\xi 
& =  \eps \, \left(B_0(u^\eps,v^\eps)u^\eps_\xi\right)_\xi,
\\
-\xi v^\eps_\xi 
& = \eps^p\ v^\eps_{\xi\xi},
\endaligned
\ee
supplemented with Riemann initial data
\be
\label{BC}
\aligned 
&   u^\eps(-\infty)=u_L,
&&  \qquad u^\eps(+\infty)=u_R,
\\
&   v^\eps(-\infty)=-1,
&& \qquad  v^\eps(+\infty)=1.
 \endaligned
\ee
In \eqref{eq5}, the maps $A_0$ and $A_1$ are now smooth {\sl scalar-valued} functions, which satisfy the following 
consistency condition with the underlying hyperbolic coupling problem determined by the functions $\gamma_\pm$ and $f_\pm$:  
there exist constants $c_1,c_2,c_3$, such that 
\be 
\label{Assum1a}
   0 < c_1 \leq A_0(u,v), \qquad   0 < c_2 \leq B_0(u,v) \leq c_3,
\ee
and 
\be
\label{Assum1b}
\aligned
&   A_0(u,-1)=\lpcL'(u), 
&&   A_0(u,1)=\lpcR'(u),
\\
&   A_1(u,-1)=\left(f_-\circ\lpcL\right)'(u),
&&   A_1(u,1)=\left(f_+\circ\lpcR\right)'(u). 
\endaligned
\ee
We set $\calU := [\min(u_L,u_R),\max(u_L,u_R)]$, and introduce the Lipschitz constants
$\omega_0$, $\omega_1$ of $A_0, A_1$, respectively, i.e. 
$$
\left|A_j(\ut,\vt)-A_j(u,v)\right| \leq \omega_j \, (|\ut-u| + |\vt- v|) 
$$
for all $(\ut,\vt),(u,v)\in\calU\times[-1,1]$ and $j=0,1$.

The first equation in \eqref{eq5} can be equivalently rewritten as 
\be
\label{eq4}
   \big(-\xi + \lambda(u^\eps,v^\eps)\big)
   G(u^\eps,v^\eps) \, B_0(u^\eps,v^\eps)u^\eps_\xi
   = 
   \eps^{\phantom{p}} \big(B_0(u^\eps,v^\eps)u^\eps_\xi\big)_\xi,
\ee
where $\lambda$ and $G$ are defined by 
$$
 \lambda(u,v) := \dfrac{A_1(u,v)}{A_0(u,v)},
 \quad \qquad 
 G(u,v) := \dfrac{A_0(u,v)}{B_0(u,v)}.
$$
Furthermore, our assumptions  
imply that (for some $\Lambda>0$)  
$$
|\lambda(u,v)| \leq \Lambda,
\qquad (u,v)\in\calU\times[-1,1], 
$$
which expresses the property of finite speed of propagation for the underlying hyperbolic problem.

Given any $M > \Lambda$, we will study first the problem on the bounded interval $[-M,M]$ with the boundary conditions in \eqref{BC}
imposed at the end points $\pm M$. Later, we will let $M$ tend to infinity.

\begin{proposition}[Existence for Riemann problem with diffusion]
\label{UnifTVB}
For each $\eps>0$ the problem \eqref{eq5} admits a smooth solution 
$(u^\eps, v^\eps) \in  C^0\big([-M,M], \calU \times [-1,1] \big)$ (space of continuous functions) 
given by the implicit formula: 
\be
\label{eq:Repres}
\aligned
& u^\eps(\xi) = u_L + (u_R-u_L)
   \frac{\begin{displaystyle}
       \int_{-M}^\xi
       e^{-h^\eps(u^\eps;\zeta)/\eps}\ B_0(u^\eps,v^\eps)^{-1}\,d\zeta
     \end{displaystyle}
   }{
     \begin{displaystyle}
       \int_{-M}^M
       e^{-h^\eps(u^\eps;\zeta)/\eps}\ B_0(u^\eps,v^\eps)^{-1}\,d\zeta
     \end{displaystyle}
   },
   \\ 
   & v^\eps(\xi)=-1+2\dfrac{\displaystyle\int_{-\infty}^\xi e^{-\zeta^2/{2\eps^p}}\,d\zeta
     }{\displaystyle\int_{-\infty}^{+\infty} e^{-\zeta^2/{2\eps^p}}\,d\zeta },
\endaligned
\ee
with 
$$
   h^\eps(u^\eps;\xi) := \int_\alpha^\xi
   \big(\zeta - \lambda\left(u^\eps,v^\eps\right)\big)\
      G(u^\eps,v^\eps)
   \, d\zeta.
$$
Moreover, these solutions $u^\eps$ and $v^\eps$ are monotone, bounded, and
 continuous, and have uniformly bounded total variation: 
$$
TV(u^\eps)\leq |u_R-u_L|,\qquad  TV(v^\eps)\leq 2.
$$
\end{proposition}

\begin{proof} Solving the second equation in \eqref{eq5} is immediate. On the other hand, 
we can rewrite the problem \eqref{eq5} as
 $$
   \begin{aligned}
     B_0(u^\eps,v^\eps) u^\eps_\xi& =\varphi,\\
     \big(-\xi + \lambda(u^\eps,v^\eps)\big)\ G(u^\eps,v^\eps)\
     \varphi&=\eps \varphi_\xi.
   \end{aligned}
 $$
Given $\ut\in  C^0\big([-M,M], \calU \big)$ we consider the solution $u^\eps(\ut;\xi)$ of the ``linearized'' problem
 $$
   \begin{aligned}
     B_0(\ut,v^\eps) u^\eps_\xi& =\varphi,\\
     \big(-\xi + \lambda(\ut,v^\eps)\big)\ G(\ut,v^\eps)\
     \varphi&=\eps \varphi_\xi,
   \end{aligned}
 $$
 together with the boundary conditions
 $$
     u^\eps(-M)=u_L,\qquad 
     u^\eps(M)=u_R.
 $$
 The solution is explicitely given by 
 \be
   \label{eq18}
   \begin{aligned}
     &u^\eps(\ut;\xi)=u_L+(u_R-u_L) \, \dfrac
     {\displaystyle\int_{-M}^\xi \varphi(\ut;\zeta)B_0(\ut,v^\eps)^{-1}\,d\zeta}
     {\displaystyle\int_{-M}^M \varphi(\ut;\zeta)B_0(\ut,v^\eps)^{-1}\,d\zeta},\\
     &\varphi(\ut;\xi) = \exp\big(-h^\eps(\ut;\xi)/\eps\big),\\
     &h^\eps(\ut;\xi) = \int_\alpha^\xi 
     \big(\zeta -
     \lambda\left(\ut,v^\eps\right)\big)\,
     G (\ut,v^\eps)
     \, d\zeta,
   \end{aligned}
 \ee
in which $\alpha\in[-M,M]$ is arbitrary.  The above formulas determine a map $T^\eps$ that takes $\ut\in  C^0\big([-M,M], \calU\big)$ 
to the function $u^\eps(\ut;\cdot)\in C^0\big([-M,M], \calU\big)$. We need to find a fixed point of $T^\eps$. 

 The uniform bounds on $\lambda(\ut,v)$ and $G(\ut,v)$ (for any $\ut\in C^0\big([-M,M], \calU\big)$ and $v\in[-1,1]$)
 allow us to choose $\alpha_\eps\in[-M,M]$ so that
 $$
 h^\eps(\ut;\xi)\geq 0, \quad \xi\in[-M,M];  
 \qquad 
 h^\eps(\ut;\alpha_\eps)=0. 
 $$
Consequently, for all $\xi\in[-M,M]$ and for some constant $c_4$ we have 
 $$
   0 \leq  h^\eps(\ut;\xi)  \leq c_4,
 $$
 so that
 $$
   c_3^{-1}\exp\left(-c_4/\eps\right)\leq \varphi(\ut,\xi)B_0(\ut,v^\eps)^{-1} \, d\xi \leq c_2^{-1}. 
 $$
We also obtain the uniform bound
 $$
   \begin{aligned}
     \left|\dfrac{d}{d\xi} u_\eps(\ut;\xi)\right|&
     \leq|u_R-u_L| \, \dfrac
     {\displaystyle \varphi(\ut;\xi)B_0(\ut,v^\eps)^{-1}}
     {\displaystyle \int_{-M}^M\varphi(\ut;\zeta)B_0(\ut,v^\eps)^{-1} \,d\zeta}\\
     &\leq |u_R-u_L| \, 
     \dfrac{c_2^{-1}}{2Mc_3^{-1}\exp\left(-c_4/\eps\right)}
     \\
     & \leq
     |u_R-u_L| \, \dfrac{c_3\ \exp(c_4/\eps)}{2Mc_2}.
   \end{aligned}
 $$
 The bound above being independent of $\ut$, we deduce that the family $T_\eps$ is equicontinuous
and its image is relatively 
 compact in $C^0\big([-M,M], \calU\big)$. Since this image is a convex closed subset of the Banach
 space $C^0([-M,M])$, Schauder's fixed point theorem applies and
 ensures that $T^\eps$ admits a fixed point. Hence, there exists
 $u^\eps\in C^0\big([-M,M], \calU\big)$ such that  $T^\eps(u^\eps)=u^\eps$, and 
 the representation formula \eqref{eq:Repres} holds. The uniform total variation bounds follow directly from \eqref{eq:Repres}.
\end{proof}


\subsection{Passage to the limit}

Using the notation in Proposition~\ref{UnifTVB}, we continue with the following two lemmas.

\begin{lemma}[Existence of a pointwise limit]
\label{L:Riemann}
After extracting a subsequence if necessary, the sequence $u^\eps$ converges pointwise
to a limiting function $u$ lying in the 
space $BV([-M,M])$ (of all functions with bounded variation): 
$$
u^\eps(\xi)\to u(\xi),\qquad \xi\in[-M,M],
$$
which satisfies, in the sense of distributions,
 \be
   \label{VarCL}
\aligned
&     -\xi\ \dfrac{d}{d\xi}\lpcL(u) +
     \dfrac{d}{d\xi} f_-(\lpcL(u))=0,  \qquad \xi <0, 
     \\
&     -\xi\ \dfrac{d}{d\xi} \lpcR(u) +
     \dfrac{d}{d\xi} f_+(\lpcR(u))=0, \qquad \xi >0.
   \endaligned
\ee
\end{lemma}

\begin{lemma}[Entropy inequalities]
 \label{L:Entropy}
The limit $u$ given by Lemma~\ref{L:Riemann} also satisfies, in the sense of distributions, 
 \be
   \label{VarEntrop}
   \aligned
     -\xi\ \dfrac{d}{d\xi} \eta(\lpcL(u)) +
     \dfrac{d}{d\xi} q_-(\lpcL(u)) \leq 0, \qquad \xi <0, 
     \\
     -\xi\ \dfrac{d}{d\xi} \eta(\lpcR(u)) +
     \dfrac{d}{d\xi} q_+(\lpcR(u)) \leq 0, \qquad \xi >0,  
   \endaligned
\ee
for all convex entropy functions $\eta$ and associated entropy flux $q_\pm'=\eta' f_\pm'$.
\end{lemma}

\begin{proof}[Proof of Lemma \ref{L:Riemann}]
In view of Lemma~\ref{UnifTVB}, Helly's compactness theorem applies and, as $\eps \to 0$, 
ensures the existence of a pointwise limit $(u,v)$ with bounded variation. 
Fix $\theta in (0,M)$ and let
 $\phi\in C^\infty_0((\theta,M))$ be a compactly supported test-function. 
 In the integral form,  \eqref{eq5} becomes 
$$
-\int_0^M \xi\ A_0(u^\eps,v^\eps)u^\eps_\xi\ \phi \, d\xi 
+ \int_0^M  A_1(u^\eps,v^\eps)u^\eps_\xi\ \phi \, d\xi 
= \eps \int_0^M \left(B_0(u^\eps,v^\eps)u^\eps_\xi\right)_\xi\phi \, d\xi,
$$
that is 
 $$
 -\int_0^M \xi\ A_0(u^\eps,1)u^\eps_\xi\ \phi \, d\xi 
 + \int_0^M  A_1(u^\eps,1)u^\eps_\xi\ \phi \, d\xi 
 + \Omega^\eps
 = \eps \int_0^M \left(B_0(u^\eps,v^\eps)u^\eps_\xi\right)_\xi\phi \, d\xi,
 $$
 where 
 $$
 \Omega^\eps := 
 \int_0^M \xi\ \big(A_0(u^\eps,1)- A_0(u^\eps,v^\eps)\big)u^\eps_\xi\ \phi  \, d\xi 
 -
 \int_0^M  \big(A_1(u^\eps,1)-A_1(u^\eps,v^\eps)\big)u^\eps_\xi\ \phi \, d\xi.
 $$
Using \eqref{Assum1a} and \eqref{Assum1b}, we can write 
 $$
 -\int_0^M \xi\ \dfrac{d}{d\xi}\lpcR(u^\eps)\ \phi \, d\xi 
 + \int_0^M  \dfrac{d}{d\xi} f_+(\lpcR(u^\eps))\ \phi \, d\xi 
 + \Omega^\eps 
 = \eps \int_0^M \left(B_0(u^\eps,v^\eps)u^\eps_\xi\right)_\xi\phi \, d\xi. 
 $$

 The term $\Omega^\eps$ vanishes with $\eps$, since
  $$
 \begin{aligned}
   |\Omega^\eps|
   & \leq M\  \int_\theta^M \omega_0|1-v^\eps|\
   |u_\xi^\eps|\ |\phi|  \, d\xi  
   +\ \int_\theta^M\omega_1
   |1-v^\eps|\ |u_\xi^\eps|\ |\phi|  \, d\xi 
   \\
   & \leq (M\ \omega_0+\omega_1)\ |1-v^\eps(\theta)|\ \|\phi\|_\infty TV(u^\eps),
 \end{aligned}
 $$
 where the total variation term $TV(u^\eps)$
 remains bounded and $|1-v^\eps(\theta)|$ tends to 0. 
 On the other hand, we have 
 $$
 \aligned
 \left|\eps\int_0^M\left(B_0(u^\eps,v^\eps)u^\eps_\xi\right)_\xi \phi  \, d\xi \right| 
 & = \left|
   \eps\int_0^M\left(B_0(u^\eps,v^\eps)u^\eps_\xi\right)\
   \phi_\xi  \, d\xi \right| 
\\
   & \leq \eps\ 
  \|\phi_\xi\|_\infty\ c_3 \, TV(u^\eps), 
\endaligned
 $$
which also converges to zero. As $\eps$ tends to $0$, we conclude that 
 $$
 \int_0^M \left(-\xi\ \dfrac{d}{d\xi}\lpcR(u) +
   \dfrac{d}{d\xi} f_+(\lpcR(u))\right)\phi  \, d\xi = 0, 
 $$
 which is the first condition in \eqref{VarCL}.  
 The same arguments apply on the interval $[-M,0]$, by using test-functions supported 
 in the interval $(-M,\theta)$, with $\theta<0$.  
\end{proof}


\begin{proof}[Proof of Lemma \ref{L:Entropy}] 
Fix $\theta >0$ and let $\phi\in C^\infty_0((\theta,M))$ be a non-negative test-function.
 Multiplying \eqref{eq5} by $\eta'(C(u^\eps,1)) \phi$, we get  
\be
\label{555}
 \aligned
& -\int_0^M \xi
 A_0(u^\eps,v^\eps) u^\eps_\xi \eta'(C(u^\eps,1))\
 \phi  \, d\xi 
 + \int_0^M
 A_1(u^\eps,v^\eps) u^\eps_\xi \eta'(C(u^\eps,1))\
 \phi  \, d\xi 
\\
& =
 \eps \int_0^M \big(B_0(u^\eps,v^\eps)u^\eps_\xi\big)_\xi\eta'(C(u^\eps,1))\
 \phi  \, d\xi. 
\endaligned
\ee
Observing that 
 $$
 \aligned
& \left|\int_0^M \xi\left(A_0(u^\eps,v^\eps) - A_0(u^\eps,1)\right)
   u^\eps_\xi \eta'(C(u^\eps,1))\
   \phi  \, d\xi \right|
   \\
   &\leq M\ \omega_0 |1-v^\eps(\theta)|\ TV(u^\eps)\
 \|\eta'(C(\cdot,1))\|_\infty \|\phi\|_\infty, 
 \endaligned
 $$
and similarly for the coefficient $A_1$, 
we see that the left-hand side of \eqref{555} is equivalent
(modulo terms that tend to zero with $\eps$) to 
 $$
 \begin{aligned}
   &-\int_0^M \xi
   \del_u(C(u^\eps,1)) \, u_\xi^\eps \eta'(C(u^\eps,1))\ \phi  \, d\xi 
   + \int_0^M A_1(u^\eps,1) \, u_\xi^\eps \eta'(C(u^\eps,1))\ \phi \, d\xi 
   \\
   & = -\int_0^M \xi \dfrac{d}{d\xi} \eta(\lpcR(u^\eps))\ \phi \, d\xi 
   + \int_0^M \dfrac{d}{d\xi} q_+(\lpcR(u^\eps))\ \phi \, d\xi.
 \end{aligned}
 $$
  On the other hand, the right-hand side of \eqref{555} 
 can be rewritten in the form
 $$
 \begin{aligned}
   &\eps\int_0^M
   \big(B_0(u^\eps,v^\eps)u^\eps_\xi\big)_\xi\eta'(C(u^\eps,1))\
   \phi  \, d\xi 
   \\
   &=
   - \eps \int_0^M
   B_0(u^\eps,v^\eps)\left(u_\xi^\eps\right)^2A_0(u^\eps,1)\eta''(C(u^\eps,1))\ \phi  \, d\xi 
\\
& \quad  - \eps\int_0^M B_0(u^\eps,v^\eps)u^\eps_\xi \eta'(C(u^\eps,1))\
   \phi_\xi  \, d\xi,
 \end{aligned}
 $$
in which the first term is non-positive and the second one tends to 0. 
Thus, letting $\eps \to 0$ we obtain 
 $$
 -\int_0^M \xi \dfrac{d}{d\xi} \eta(\lpcR(u))\ \phi \, d\xi 
 + \int_0^M \dfrac{d}{d\xi} q_+(\lpcR(u))\ \phi \, d\xi 
 \leq 0,
 $$
which yields the second identity in the statement of the lemma. The 
derivation of the first identity in the half-space $\xi<0$ is completely similar.
\end{proof}


\subsection{Riemann problem for the hyperbolic coupling problem}

In view of the boundary condition \eqref{BC}, it is natural to extend $u$ by 
\be
\label{Recover}
   u(\xi)=
   \begin{cases}
     u_L, & \xi\leq -M,\\
     u_R, & \xi\geq M.
 \end{cases}
\ee
The conclusions in Lemmas \ref{L:Riemann} and \ref{L:Entropy} then clearly hold on the intervals 
$(-\infty,-M)$ and $(M,+\infty)$. In addition, an interface condition for the solution $u$ at the end points 
$\xi=-M$ and $\xi=M$ is now derived, which 
is necessary to ensure that \eqref{VarCL} and \eqref{VarEntrop} extend 
(in the sense of distributions) to $(0,+\infty)$ and $(-\infty,0)$. 

\begin{lemma}
\label{L:Rrecov}
As $\eps$ tends to $0$, the solution $u^\eps$ converges uniformly toward $u_R$ (repectively $u_L$) on the interval
 $(\Lambda,M]$ (resp. $(-M,-\Lambda[$).
\end{lemma}

\begin{proof}
Let $\xi\in\left({\Lambda+M\over 2},M\right)$ be given. 
According to \eqref{eq:Repres} we have
$$ 
 \aligned
 |u^\eps(\xi)-u_R|
 & =|u_L-u_R| \, 
 {\displaystyle\int_\xi^M  e^{-h^\eps(u^\eps)/\eps}B_0(u^\eps,v^\eps)^{-1} \, d\xi 
     \over \displaystyle\int_{-M}^M e^{-h^\eps(u^\eps)/\eps}B_0(u^\eps,v^\eps)^{-1}\, d\xi }
     \\
 & \leq |u_L-u_R| \, 
 {\displaystyle\int_{\Lambda+M\over 2}^M  e^{-h^\eps(u^\eps)/\eps}B_0(u^\eps,v^\eps)^{-1}\, d\xi 
    \over \displaystyle\int_{-M}^M e^{-h^\eps(u^\eps)/\eps}B_0(u^\eps,v^\eps)^{-1}\, d\xi }.
\endaligned
$$
We use here the constant $\alpha\in[-\Lambda,\Lambda]$ as the lower
integration bound for both $h^\eps$ and the function ($\xi \in \RR$) 
$$
h(\xi) := \int_\alpha^\xi\big(\zeta-\lambda(u(\zeta),v(\zeta))\big)G(u(\zeta),v(\zeta))\,d\zeta, 
$$
hence $h \geq 0$. Moreover, $h^\eps(u^\eps,\cdot)$ converges uniformly to $h$, with 
$$
\begin{aligned}
&\big|h^\eps(u^\eps,\xi)-h(\xi)\big|
\\
& =\left|\int_\alpha^\xi
 \big(\lambda(u^\eps(\zeta),v^\eps(\zeta))-\lambda(u(\zeta),v(\zeta))\big)G(u^\eps(\zeta),v^\eps(\zeta))\,d\zeta\right|\\
& \leq \left({1\over c}\omega_1+\|A_1\|_\infty{\omega_0\over
   c^2}\right)\|G\|_\infty\int_{-M}^M\big(|u^\eps(\zeta)-u(\zeta)|+|v^\eps(\zeta)-v(\zeta)|\big)\,d\zeta\\ 
& \leq \left({1\over c}\omega_1+\|A_1\|_\infty{\omega_0\over
   c^2}\right)\|G\|_\infty\big(\|u^\eps-u\|_{L^1} + \|v^\eps-v\|_{L^1} \big). 
\end{aligned}
$$
The uniform convergence of $h^\eps(u^\eps,\cdot)$ toward a positive continuous function $h$ such that $h(\alpha)=0$
insures that there exists $\eps_0>0$ together with $A>B>0$ and $\eta>0$
such that, for all $\eps<\eps_0$,
$$
\begin{aligned}
& h^\eps(u^\eps,\xi)\geq A,\qquad {\Lambda+M\over 2}\leq\xi\leq M, 
\\
& h^\eps(u^\eps,\xi)\leq B,\qquad \xi\in[-M,M],\ |\xi-\alpha| \leq \eta.
\end{aligned}
$$
Thus, we deduce that 
$$
\aligned 
|u^\eps(\xi)-u_R|
& \leq |u_R-u_L|{{M-\Lambda\over
   2}e^{-A/\eps}c_2^{-1}\over\eta e^{-B/\eps}c_3^{-1}}
   \\
   & =
|u_R-u_L|{M-\Lambda\over 2\eta}{c_3\over c_2}e^{-(A-B)/\eps},
\endaligned
$$
and $u^\eps$ converges uniformly toward $u_R$ on
the open interval $({\Lambda+M\over 2},M)$.
The same argument leads to the uniform convergence of $u^\eps$ toward $u_L$ on $(-M,{\Lambda-M\over 2})$.
\end{proof}

We summarize the results established in Lemmas \ref{L:Riemann}, \ref{L:Entropy}, and \ref{L:Rrecov}, 
as follows.

\begin{theorem}[The Riemann problem for the coupling of two scalar equations]
Up to extracting a subsequence, the solution $u^\eps$ to \eqref{eq5}-\eqref{BC}
converges pointwise to a function $u\in BV(\RR)$, 
$$
   u^\eps(\xi)\to u (\xi),\qquad \xi\in{\RR}, 
$$
which satisfies the conservation laws and entropy inequalities  
 $$
\begin{aligned}
     -\xi\ \dfrac{d}{d\xi}\lpcL(u) +
     \dfrac{d}{d\xi} f_-(\lpcL(u))=0,\qquad \xi<0,
     \\
     -\xi\ \dfrac{d}{d\xi}\lpcR(u) +
     \dfrac{d}{d\xi} f_+(\lpcR(u))=0,\qquad \xi>0,
\end{aligned}
$$
and 
$$
\begin{aligned}
     -\xi\ \dfrac{d}{d\xi} \eta(\lpcL(u)) +
     \dfrac{d}{d\xi} q_-(\lpcL(u)) \leq 0,\qquad \xi<0, 
     \\
     -\xi\ \dfrac{d}{d\xi} \eta(\lpcR(u)) +
     \dfrac{d}{d\xi} q_+(\lpcR(u)) \leq 0,\qquad \xi>0, 
\end{aligned}
$$
for all convex entropy pairs, together with the boundary conditions
 $$
     u(-\infty)=u_L,\qquad 
     u(+\infty)=u_R.
 $$ 
\end{theorem}

Equivalently, in terms of the function $w$ in \eqref{ChangeVariable}, we have established
 $$
\begin{aligned}
     -\xi \dfrac{d}{d\xi} w + \dfrac{d}{d\xi} f_-(w) =0,\qquad -\xi \dfrac{d}{d\xi} \eta(w) + \dfrac{d}{d\xi} q_-(w) \leq 0,\qquad \xi<0,
     \\
     -\xi \dfrac{d}{d\xi} w + \dfrac{d}{d\xi} f_+(w) =0,\qquad -\xi \dfrac{d}{d\xi} \eta(w) + \dfrac{d}{d\xi} q_+(w) \leq 0,\qquad \xi>0, 
\end{aligned}
 $$
 with 
 $$
     w(-\infty)=w_L,\qquad
     w(+\infty)=w_R.
 $$
We have thus established that the interface problem admits a solution which has bounded variation. 

 
\section{Existence theory for systems}
\label{section4}

\subsection{Terminology and notation}

We will now generalize the technique developed in Tzavaras~\cite{Tzavaras96} and Joseph and LeFloch~\cite{JosephLeFloch99,JosephLeFloch5}, 
and cover the class of nonconservative and resonant systems \eqref{1.3} 
under consideration. We follow closely the notation and presentation in \cite{JosephLeFloch5}. 

Specifically we consider the {\sl diffusive Riemann problem} \eqref{1.3} with Riemann data $u_L, u_R$,
and establish that, provided $u_L, u_R \in \calB(\delta_1)$ with
a sufficiently small $\delta_1 < \delta_0$ and under some structural hypotheses on the matrix fields $A_0, A_1$, 
this problem admits a smooth, self-similar solution,  
 $u_\eps = u_\eps(x/t) \in \calB(\delta_0)$
and $v_\eps = v_\eps(x/t) \in [-1,1]$, which has uniformly bounded total variation
\be
\label{1.5} 
TV(u_\eps) + TV(v_\eps) \leq C, 
\ee 
for some uniform $C>0$. 
Solutions to \eqref{1.3} will be decomposed in terms of ``wave strengths'' of the associated Riemann 
problem \eqref{eq:UV}. The uniform estimate \eqref{1.5} is the key to the convergence analysis as $\eps \to 0$,
 and the proof of the existence of the Riemann solution to the underlying hyperbolic problem, discussed in the following section. 

We are interested in solutions $u$ taking values in a
small neighborhood of a given state (normalized to be the origin without loss of
generality), that is, in the ball $\calU :=
\calB(\delta_0)$ with (small) radius $\delta_0$.
For each $u \in \calB(\delta_0)$ and $v\in [-1,1]$, let 
$\lam_1(u,v) < \ldots < \lambda_N(u,v)$ 
be the real and distinct eigenvalues of the $N\times N$ matrix
$$
A(u,v) := A_1(u,v) \, A_0(u,v)^{-1},
$$ 
and let $l_1(u,v),\ldots, l_N(u,v)$ and $r_1(u,v),\ldots, r_N(u,v)$ be basis of left- and
right-eigenvectors, respectively, normalized so that
$l_i(u,v) \cdot r_j(u,v) = 0$ if $i \neq j$ and $l_i(u,v) \cdot
r_i(u,v) = 1$.

By reducing $\delta_0$ if necessary, we may assume that the wave
speeds $\lam_i(u,v)$ are sufficiently close to the constants
$\lam_i(0,0)$ and, in particular, are uniformly separated for all $u
\in \calB(\delta_0)$ in the sense that, for some constants
$$
- M < \bLam_1 < \Lamb_1 < \bLam_2 < \ldots < \bLam_N < \Lamb_N < M, 
$$
\be
 \label{2.1}
 \bLam_i : = \lam_i(0,0) - O(\delta_0),  
 \quad \Lamb_i : = \lam_i(0,0) + O(\delta_0) 
\ee
and
\be
 \label{2.2}
 \bLam_i \leq \lam_i(u,v) \leq \Lamb_i, \quad u \in \calB(\delta_0), v \in [-1,1].
\ee
Let $m$ be the index associated with the resonant wave,
i.e.~such that $\lambda_m$ may change sign. 
In addition, for $\delta_0$ sufficiently small the vectors $r_i(u,v)$ are sufficiently close to $r_i(0,0)$
and we assume that
\be
 \label{2.3}
 \aligned
 & l_i (u_1,v) \cdot r_i (u_2,v) \geq 1 - \delta_0,
 \quad u_1, u_2 \in \calB(\delta_0), \, v \in [-1,1],
 \\ 
 & |l_i(u_1,v) \cdot r_j(u_2,v)| \leq \delta_0,
 \quad u_1, u_2 \in \calB(\delta_0), \, i \neq j, \, v \in [-1,1].
 \endaligned  
\ee

In \eqref{1.3}, the matrix $B_0=B_0(u,v)$ is assumed to be non-degenerate and 
depend smoothly upon $u$ and $v$. We treat the case that the diffusion
matrix $B(u,v) := B_0(u,v) \, A_0(u,v)^{-1}$ is sufficient close to
the identity matrix, that is, for some given matrix norm and $\eta>0$ sufficiently small, 
\be
\label{556}
\sup_{\substack{u \in \calB(\delta_0)\\v \in [-1,1]}} \left|B(u,v) - I\right| \leq \eta. 
\ee
To handle arbitrary diffusion matrices $B(u,v)$, we follow Joseph and LeFloch~\cite{JosephLeFloch07} 
and introduce 
the 
{\sl generalized eigenvalue problem:} 
\be
 \label{2.4}
 \aligned
 & \bigl( - \xi\textrm{ Id} + A(u,v) \bigr) \, \hatr_i(u,v,\xi) = \mu_i(u,v,\xi) \, B(u,v) \,
 \hatr_i(u,v,\xi),\\
 & \hatl_i(u,v,\xi) \cdot \bigl( - \xi\textrm{ Id} + A(u,v) \bigr) = \mu_i(u,v,\xi) \, \hatl_i(u,v,\xi) \cdot B(u,v). 
 \endaligned  
\ee
with unknowns the vectors $\hatr_i(u,v,\xi), \hatl_i(u,v,\xi)$ and the scalars $\mu_i(u,v,\xi)$. 
We impose the following normalization to generalized left- and right-eigenvectors: 
$$
 \aligned 
 & \hatr_i(u,v,\xi) \cdot \hatr_i(u,v,\xi) = 1, 
 \\ 
 & \hatl_i(u,v,\xi) \cdot B(u,v) \, \hatr_j(u,v,\xi) = 0 \quad \text{ if } i \neq j,
 \\
 &\hatl_i(u,v,\xi) \cdot B(u,v) \, \hatr_i(u,v,\xi) = 1.
 \endaligned 
$$
Multiplying the first equation in \eqref{2.4} on the left by $\hatr_i(u,v,\xi)$ 
and rearranging terms, we get
\be
 \label{2.5}
 \mu_i(u,v,\xi) =  \bigl( -\xi +\hatlam_i(u,v,\xi) \bigr)\, d_i(u,v,\xi), 
\ee
where
\be
 \label{2.6}
 \aligned
 & \hatlam_i(u,v,\xi) : = \hatr_i(u,v,\xi) \cdot A(u,v) \, \hatr_i(u,v,\xi),
 \\
 & 1/d_i(u,v,\xi) : = \hatr_i(u,v,\xi) \cdot B(u,v) \, \hatr_i(u,v,\xi).
 \endaligned 
\ee
Clearly, in the special case where $B(u,v) = I$, we find
$$
 \mu_i(u,v,\xi) = - \xi + \lam_i(u,v), 
 \quad \hatr_i(u,v,\xi) = r_i(u,v), 
 \quad \hatl_i(u,v,\xi) = l_i(u,v).
$$
So, by continuity, when $B$ gets closer to the identity matrix, the coefficients 
$d_i(u,v,\xi)$ and $\hatlam_i(u,v,\xi)$ get closer to $1$ and $\lam_i(u,v)$, respectively. 
In consequence, under the assumption $|B(u,v) - I| < \eta$ with $\eta$ sufficiently small
and by increasing the gaps $\Lamb_i - \bLam_i$ if necessary, we can always 
assume that
$$
\aligned
\bLam_i- O(\eta) &\leq \hatlam_i(u,v,\xi) \leq \Lamb_i+ O(\eta),
\\
1- O(\eta)  &\leq d_i(u,v,\xi) \leq 1 + O(\eta) 
\endaligned
$$
for $u \in \calB(\delta_0), \, v\in [-1,1], \, \xi \in [-M,M]$. 
The following property was pointed out in \cite{JosephLeFloch07}.

\begin{lemma}
The $\xi$-derivatives of the generalized eigenvectors and eigenvalues satisfy
\be
\label{2.8}
\begin{aligned}
 |\del_\xi \hatr_i(u,v,\xi)| &= O(\eta),
 \\
 \del_\xi \mu_i(u,v,\xi) &= -1 + O(\eta).
\end{aligned}
\ee
\end{lemma}

Finally, we introduce the coefficient 
\be
\label{nu}
 \nu := \sup \left|\hatl_i\cdot\del_v(B\hatr_j)\right|
\ee 
allows one to measure how closed the left-hand and right-hand hyperbolic models are and, from now on, 
this coefficient is assumed to be  
sufficiently small.

\begin{example}[$p$-system] 
The following example illustrates the meaning of $\nu$.
Consider the coupling between two systems of two conservation laws, specifically
$p$-systems, with two different pressure laws $p_\pm=p_\pm(\tau)$
$$
\begin{aligned}
\del_t\tau-\del_x V &=0,\\
\del_t V + \del_xp_{\pm}(\tau)&=0, \qquad \qquad \tau>0, \quad V \in \RR, 
\end{aligned}
$$
in which the associated flux $F_\pm$ have Jacobian matrices 
$$
\nabla F_\pm=\begin{pmatrix}
          0 & -1\cr p_\pm'(\tau) & 0
         \end{pmatrix}.
$$
Suppose that the coupling is based on $B=\mathrm{Id}$ and the average matrix 
$$
A(\tau,v)=\frac{1+v}{2}\nabla F_++\frac{1-v}{2}\nabla F_-= \begin{pmatrix}
        0 & -1 \cr
	\frac{1+v}{2} p_+'(\tau)+ \frac{1-v}{2} p_-'(\tau) & 0
       \end{pmatrix}.
$$
Then, the eigenvalues and eigenvectors of the system are ($j=1,2$ and $\pm$ 
corresponding to the two wave families)
$$
\begin{aligned}
 \lambda_j(\tau,v)&=\pm\sqrt{-\frac{1+v}{2} p_+'(\tau)- \frac{1-v}{2}p_-'(\tau)},\\
 \hatr_j(\tau,v)&=\begin{pmatrix}
                1, & \mp\sqrt{-\frac{1+v}{2} p_+'(\tau)- \frac{1-v}{2}p_-'(\tau)} \, 
               \end{pmatrix},
               \\
\hatl_j(\tau,v)&=\begin{pmatrix}
             \pm\sqrt{-\frac{1+v}{2} p_+'(\tau)- \frac{1-v}{2}p_-'(\tau)}, & 1
            \end{pmatrix}, 
\end{aligned}
$$
and a tedious calculation yields 
$$
\aligned
\left|\hatl_i\cdot\del_v(B\hatr_j)\right|
& \leq 
\dfrac{|p_+'(\tau)-p_-'(\tau)|}{2\sqrt{-\frac{1+v}{2} p_+'(\tau)- \frac{1-v}{2}p_-'(\tau)}}
\\
& \leq 
\dfrac{|p_+'(\tau)-p_-'(\tau)|}{2\min\Big(\sqrt{-p_+'(\tau)},\sqrt{-p_-'(\tau)}\Big)}, 
\endaligned
$$
so that 
$$\nu \leq C \sup |p'_+-p'_-|.$$
\end{example}


\subsection{Equations satisfied by the characteristic coefficients}

We supplement (cf.~\eqref{1.3})
\be
\label{1.3-NEW} 
\begin{aligned}
   \big(-\xi A_0(u^\eps,v^\eps)+ A_1(u^\eps,v^\eps)\big)u^\eps_\xi &= \eps
   \bigr(B_0(u^\eps,v^\eps)u^\eps_\xi\bigr)_\xi,
   \\
   -\xi v^\eps_\xi &= \eps^p v^\eps_{\xi\xi}
\end{aligned}
\ee 
with the following boundary conditions, inherited
from Riemann initial data
\be
 \label{2.10}
 \aligned
 & u^\eps(- M) = u_L, 
 \quad
  u^\eps(M) = u_R,  \\
 & v^\eps(- M) = -1,
 \quad
  v^\eps(M) = 1.
 \endaligned
\ee
We describe now an asymptotic expansion for the solution of the diffusive Riemann
problem \eqref{1.3-NEW}-\eqref{2.10}. To handle an arbitrary diffusion matrix, the
decomposition must be based on the modified eigenvectors \eqref{2.4}. We first solve explicitly the equation concerning $v^\eps$:
\be
 \label{2.11}
\psi^\eps(\xi):=v^\eps_\xi(\xi) = 2
 {e^{-{\xi^2\over2 \eps^p}} \over \int_{- M}^{
     M}e^{-{x^2\over 2\eps^p}} dx}.
\ee

\begin{remark}
In the limiting case $p=+\infty$, we would formally get a Dirac mass solution
$\psi=v_\xi=2\delta_{\xi=0}$ and the transition from
$v_L=-1$ to $v_R=1$ at the interface would then be discontinous.
\end{remark}

The main equation \eqref{1.3-NEW} now reads
\be
\label{2.12b}
\bigl(-\xi I + A(u^\eps,v^\eps)\bigr)A_0(u^\eps,v^\eps)u^\eps_\xi=\eps\bigl(B(u^\eps,v^\eps)A_0(u^\eps,v^\eps)u^\eps_\xi\bigr)_\xi.
\ee
To deal with this equation, we introduce a decomposition of the vector $A_0(u^\eps,v^\eps) u^\eps_\xi(\xi)$ 
on the basis of eigenvectors $\hatr_j(u^\eps(\xi),v^\eps(\xi),\xi)$, that is, 
\be
 \label{2.12}
\aligned
 A_0(u^\eps,v^\eps)\,u^\eps_\xi(\xi) = & \sum_{j=1}^N  \, a_j^\eps(\xi) \, \hatr_j(u^\eps, v^\eps, \xi),
 \\
 a_j^\eps(\xi) = & \, \hatl_j(u^\eps, v^\eps, \xi) B(u^\eps,v^\eps)A_0(u^\eps,v^\eps) u^\eps_\xi(\xi),
 \endaligned
\ee
where the functions $a_j^\eps$ are referred to as the {\sl characteristic coefficients}. 
Removing the explicit dependence in $\eps$,
the right-hand side of \eqref{2.12b} takes the form 
$$
\aligned
 \bigl( B(u,v) \, A_0(u,v)u_\xi \bigr)_\xi 
 = &
 \sum_j  a_k' \, B(u,v) \, \hatr_j(u,v, \cdot)\\
 & +
 \sum_{j,k} a_j\,a_k \, D_u\bigl(B \, \hatr_j \bigr) (u,v,
 \cdot) \,  A_0(u,v)^{-1} \, \hatr_k(u,v, \cdot)\\
 & +
 \sum_j a_j \, \del_v \bigl(B \, \hatr_j \bigr) (u,v, \cdot) \, v_\xi
 +
 \sum_j  a_j \, B(u,v) \, \del_\xi \hatr_j  (u,v, \cdot).
\endaligned
$$
Therefore, given any solution to \eqref{1.3-NEW}, we obtain 
$$
\aligned
& \sum_j  \Big( \eps \, a_j' \, B(u,v) \, \hatr_j(u,v, \cdot)
- a_j \, \bigl( -\xi + A(u,v) \bigr) \, \hatr_j(u,v, \cdot) \Big)
\\
& = - \eps \, \sum_{j,k}  \, a_j \,a_k \,
D_u \bigl( B \, \hatr_j \bigr)(u,v, \cdot) \, A_0(u,v)^{-1} \, \hatr_k(u,v, \cdot)
-
\eps \sum_j a_j \, \del_v \bigl(B \, \hatr_j \bigr) (u,v, \cdot) \, \psi\\
& \quad - \eps \sum_j  a_j \, B(u,v) \,  \del_\xi \hatr_j(u,v, \cdot).
\endaligned
$$

Now, multiplying the above equations by each vector $\hatl_i(u,v, \cdot)$ for $i=1, \ldots, N$
and relying on the equation~\eqref{2.4}, we arrive at
 a {\sl coupled system of $N$ differential equations} 
for the characteristic coefficients $a_i$:
\begin{subequations}
 \label{2.13}
 \be
   \label{2.13a}
   a_i' - \frac{\mu_i(u,v, \cdot)}{\eps} \, a_i
   =
   \eta  L_i(u,v,\cdot)\,
   + 
   Q_i(u,v, \cdot) \,
   +
   S_i(u,v, \cdot),
 \ee
 where the linear, quadratic, and source terms are defined by 
 \be
   \label{2.13b}
   \aligned
   & L_i(u,v,\cdot)
   := \sum_j \pi_{ij}(u,v, \cdot) \, a_j,\qquad  Q_i(u,v, \cdot) 
   := \sum_{j, k} \kappa_{ijk}(u,v, \cdot) \, a_j \, a_k, \\
   & S_i(u,v, \cdot)
   := \sum_j \sigma_{ij}(u,v, \cdot) \, a_j \, \psi,
   \endaligned
 \ee
 respectively, with  
 \be
   \label{2.13c}
   \aligned
   & \pi_{ij}(u,v, \cdot) : = 
   -\eta^{-1} \, \big(\hatl_i \cdot B \,  \del_\xi \hatr_j \big) (u,v, \cdot), \\ 
   & \kappa_{ijk}(u,v, \cdot) : =
   - \big( \hatl_i \cdot D_u ( B \, \hatr_j ) \, A_0^{-1}\hatr_k \big) (u,v, \cdot),\\
   & \sigma_{ij}(u,v, \cdot) : =
   \big( \hatl_i \cdot \del_v ( B\, \hatr_j ) \big) (u,v, \cdot).
   \endaligned
 \ee
\end{subequations}

The main equation \eqref{1.3-NEW} is therefore completely equivalent to \eqref{2.12}-\eqref{2.13}. 
In view of \eqref{2.8} and \eqref{nu}, \eqref{2.13} takes the form 
\be
 \label{2.14}
  a_i' - \frac{1}{\eps} \mu_i(u,v, \cdot) \, a_i =
  O(\eta) \, \sum_j |a_j|
  + O(1) \, \sum_{j,k} |a_j| \, |a_k|
  + O(\nu) \, \sum_j |a_j| \, |\psi|.
\ee
Consider the principal part of \eqref{2.13},
that is, given some function $u=u(y)$ (which at this stage need not to
be a solution to \eqref{2.13}) let us consider the following {\sl decoupled
 system of $N$ linear equations}:
\be
\label{2.15}
\varphi_i^{\star}{}' - \frac{\mu_i(u,v, \cdot)}{\eps} \, \varphi_i^\star 
= 0, \qquad i=1,\ldots, N. 
\ee
The general solution is (up to a multiplicative constant) 
\be
\label{2.16}
\varphi_i^\star := \frac{e^{ - g_i/\eps }}
{\displaystyle \int_{-M}^{M}e^{ - g_i / \eps } \,dy}, 
\quad \qquad 
g_i(y) := -\int_{\rho_i}^y \mu_i(u,v, \cdot)(x) \, dx,
\ee
where the constants $\rho_i$ will be chosen so that the functions $g_i$
are non-negative (cf.~Section~\ref{subsec:fundamental}). Clearly, the functions $\varphi_i^\star$ are strictly positive 
and their integral over $\RR$ equals 1. 

We will search for the general solutions $a_i$ of \eqref{2.13a} in the form
\be
\label{development}
a_i = \tau_i \varphi^\star_i + \theta_i
\ee
where $\tau_i$ refers to a  wave strengh and with $\theta_i$ small w.r.t $\tau_i$.
When solving the equation \eqref{2.13a} for a given right-hand side and considering again first-order terms, we are naturally led to consider the following {\it linear wave coefficients} $J_{j\to i}$, the {\it quadratic wave coefficients} $F_{jk\to i}$ and the {\it resonant quadratic coefficients} $J_{j\to i}^\psi$ defined by
\be
 \label{eq:linear}
 J_{j\to i}(y) := \varphi_i^\star(y) \int_{c_i}^y { \varphi_j^\star(x)\over \varphi_i^\star(x)} \,\, dx,
\ee
\be
 \label{eq:quadra}
 F_{j,k\to i}(y) : = \varphi_i^\star(y) \int_{c_i}^y 
 \frac{\varphi_j^\star  \, \varphi_k^\star}{\varphi_i^\star} \, dx,
\ee
\be
 \label{eq:coupling}
   J_{j\to i}^\psi(y) := \varphi_i^\star(\xi)\int_{c_i}^\xi\psi(x)\dfrac{\varphi_j^\star(x)}{\varphi_i^\star(x)}\,\, dx
\ee
for some constants $c_i \in [\bLam_i, \Lamb_i]$ independent of $\eps$. By studying these coefficients, we will gain useful information on the possible growth of the total variation of solutions to \eqref{2.13}: roughly speaking, $J_{j\to i}$ bounds the influence of the $j$-th family on the $i$-th family, $F_{j,k\to i}$ bounds the contribution on the $i$-th family due to 
waves of the $j$-th and $k$-th characteristic families and $J_{j\to i}^\psi$ bounds the influence of the $j$-th family on the $i$-th family ``through'' the coupling wave $\psi$. 


\subsection{Linearized wave measures}
\label{subsec:fundamental}

The results of this section were established earlier in Joseph and LeFloch~\cite{JosephLeFloch99} and 
are presented for the convenience of the reader. We fix some speed range $[\lam_{\min},\lam_{max}]$ and
analyze the formula~\eqref{2.16} within this speed range. We introduce first the space of functions that 
are ``almost
linear at infinity''. Observe the coefficients 
$\mu_i$ will belong to such spaces in their respective speed ranges $[\bLam_i, \Lamb_i]$.

\begin{definition}
\label{Lspace}
A function $h:[-M,M] \mapsto \RR$ is said to belong to the
class $\mathcal{L}$ of almost linear functions 
if there exists two functions $d,\lam\in L^\infty([-M,M],\RR)$,
and two positive reals $d_{\min},d_{\max}$ such that 
$$ 
h(x)=d(x)\bigl(\lam(x)-x\bigr), \quad x\in[-M,M]
$$
and
$$
\aligned
0<d_{\min} &\leq d(x)\leq d_{\max},\\
-M<\lam_{\min} &\leq \lam(x) \leq \lam_{\max}<M.
\endaligned
$$
\end{definition}

\begin{lemma}
\label{l:rho}
Let $h:[-M,M]\mapsto\RR$ be a function of class $\mathcal{L}$ and, given $y\in[-M,M]$, set
$$
g(x)=-\int_y^x h(x') \,\,dx',\qquad x\in[-M,M]. 
$$
Then, $g$ is Lipschitz continuous and achieves its global minimum at some (non-unique)
point $\rho \in [\lam_{\min},\lam_{\max}]$ such that $\lam(\rho)~=~\rho$.
\end{lemma}

\begin{proof} Using definition~\ref{Lspace} of the class $\mathcal{L}$ we get
$$
\aligned
&h(x)>0,\qquad x<\lam_{\min},\\
&h(x)<0,\qquad x>\lam_{\max},
\endaligned
$$
$g$ being continuous, decreasing on $[-M,\lam_{\min}]$ and increasing
on $[\lam_{\max},M]$ (since  $g'=-h$), we deduce $g$ achieves its global minimum at
some (non-unique) point $\rho\in[\lam_{\min},\lam_{\max}]$.
Moreover $h(\rho)=-g'=0$, that means $\lam(\rho)=\rho$.
\end{proof}

We now introduce a new notation suggested by the formula \eqref{2.16} and derive 
useful algebraic properties. For $x,y \in [-M,M]$ and $h\in L^1([-M,M],\RR)$, we set 
$$
\aligned
\phi(y,x; h) &:= \exp\left( \frac{1}{ \eps}  \int_y^x h(x') \, dx' \right),\\
I(y; h) &:= \int_{-M}^M\phi(y,x'; h)\,\,dx',\qquad 
\varphi(x; h) := \frac{\phi(y,x; h)}{I(y; h)}.
\endaligned
$$
That is, $x\mapsto\varphi(x;h)$ is a solution (with unit mass) to the differential equation
$$
\varphi'- \frac{h}{\eps}\varphi = 0.
$$

\begin{lemma}[Algebraic properties of the mapping $\phi$]
 \label{l:algebraic}
 For general functions $h,\tilde h$ in $L^1([-M,M],\RR)$ and $x,y,z \in [-M,M]$
the following algebraic properties hold: 
 $$
   \aligned
   &\frac{\phi(x,y; h) }{ \phi(x,y; \tilde h)} = \phi(x,y; h - \tilde h),\\
   &\frac{\phi(x,y; h) }{ \phi(z,y; h)} = \frac{\phi(y,z; h) }{ \phi(y,x; h)}
   = \phi(x,z; h),\\
   &\phi(x,y; h) = \phi(y,x; -h) =  \frac{1 }{ \phi(y,x; h)}. 
   \endaligned
 $$
\end{lemma}

 In view of these properties, the definition of
$\varphi(x;h)$ is checked to be independent of the variable $y$.  
When the function $h$
belongs to the class $\mathcal{L}$, by taking $y=\rho$ given by Lemma~\ref{l:rho} we
obtain a negative argument in the exponential defining $\phi$ 
and that argument 
vanishes at points where $g$ is minimized.

Now, fix $\mu_1,\mu_2$ in the class $\mathcal{L}$,
 and let 
$\varphi_1:=\varphi(\cdot;\mu_1)$ and $\varphi_2:=\varphi(\cdot;\mu_2)$ be
 solutions to the 
differential equation associated with $\mu_1$ and $\mu_2$, respectively. 
Denote by $\rho_1$ and
$\rho_2$ the minimization points of the associated functions as 
defined in Lemma~\ref{l:rho}. Moreover, fix some yet unspecified scalar 
$c \in [\lam_{\min},\lam_{\max}]$ which we take to be independent of $\eps$. We then
want to control the linear coefficient
\begin{subequations} 
 \be
   \label{3.5a}
   J_{\varphi_2 \to \varphi_1}(y)=\varphi_1(y) \, \int_c^y \frac{\varphi_2(x)}{\varphi_1(x)}\,\,dx
 \ee
 characterising the first order (linear) influence of $\varphi_2$ on $\varphi_1$.
 Using Lemma~\ref{l:algebraic} we find those two following equivalent expressions, both useful in the sequel according to the sign of $\mu_2 - \mu_1$,
 $$
   \aligned
   &J_{\varphi_2 \to
     \varphi_1}(y)=\frac{I(\rho_1;\mu_1)}{I(\rho_2;\mu_2)}\phi(\rho_2,\rho_1;\mu_2)\, \varphi_1(y)\, \int_c^y
   \phi(\rho_1,x;\mu_2-\mu_1)\,\,dx,\\
   &J_{\varphi_2 \to \varphi_1}(y)=\varphi_2(y)\, \int_c^y \phi(x,y;\mu_1-\mu_2)\,\,dx.
   \endaligned
 $$
\end{subequations}
In order to estimate those coefficients, we also need for more information on
the asymptotical behavior of appearing quantities as $\eps$ tends to
$0$. The two following lemmas will give it.

\begin{lemma} [Asymptotic behavior of $\phi$]
\label{l:asymptotic}
Let $[x,y]$ be an interval of $[-M,M]$, (with $x<y$), and let $h$ be a continuous function on $[x,y]$. If $h$ is strictly positive bounded, say we have $h(x') \geq h_{\min} >0$ on $[x,y]$, then the following integral is at most linear in $\eps$
$$
\int_x^y \phi(x',y,-h) \, dx' \leq \frac{\eps }{h_{\min}},
$$
\end{lemma}

\begin{lemma}[Asymptotic behavior of $I(\rho; h)$]
\label{l:mass}
 For a function $h$ in the class $\mathcal{L}$, and $\rho$ defined in Lemma \ref{l:rho},
 the integrals $I(\rho; h)$ satisfy
$$
   c \eps \leq I(\rho;h) \leq 2M.
$$
\end{lemma}

\begin{proof}[Proof of Lemma~\ref{l:asymptotic}]
Let us suppose $h(x') \geq h_{\min} > 0$. Then, for all $x < x' < y$, we have 
$$
-\int_{x'}^y h(t)\, dt  \leq -h_{\min} \, (y - x'),
$$
thus
$$
\int_x^y \phi(x',y;h) \, dx' \leq \int_x^y e^{-h_{\min} (y-x')/\eps} \, dx'
\leq
\frac{\eps }{h_{\min}}\int_0^\infty e^{-x'} \, dx' \leq \frac{\eps }{h_{\min}},
$$
and the result follows.
\end{proof}

\begin{proof}[Proof of Lemma~\ref{l:mass}]
By the definition of $\rho$, the argument of the exponential
 defining $\phi$ is everywhere nonpositive so $\phi \leq 1$ and
 $I(\rho; h) \leq 2M$. Moreover at the point $\rho$, the primitive of
 $h$ is locally Lipschitz continuous, so there exists a sufficiently small $\eta$
 and a constant $c>0$ such that
 $$
   0 \leq -\int_\rho^x h(x') \,\,dx' \leq \frac{1}{c} |x-\rho|,\qquad
   |x-\rho|<\eta.
 $$
 Then
 $$
   I(\rho; h)
   \geq \int^{\eta}_{-\eta}e^{ -\frac{1}{c \eps} |x|} \,\, dx
   = 2c\eps \int_0^{\frac{\eta}{c\eps}} e^{-x'}\,\, dx' \geq c\eps.
 $$
\end{proof}

Note that in Lemma~\ref{l:asymptotic}, if $h_{\min}=0$, then the considered integral remains bounded as $\eps$ vanishes.\\

The following result Lemma~\ref{l:wavelocal} shows that, on any compact subset of the
complement set $[\lam_{\min},\lam_{\max}]^c$, the mass of the linearized wave measures 
tends to zero. In the limit, all the mass of the wave measure $\varphi(\cdot;h)$
is concentrated on the interval $[\lam_{\min},\lam_{\max}]$.

\begin{lemma}[Behavior of linearized wave measures] 
\label{l:wavelocal}
 For all $h$ in the class $\mathcal{L}$ the function $\varphi(\cdot;h)$ satisfies 
 the estimates 
$$
   0 \leq \varphi(x;h) \leq O(1/\eps) 
   \begin{cases} 
     e^{ - (x - \lam_{\min} )^2 d_{\min} / 2 \eps  },   & - M < x < \lam_{\min},\\
     1, & x \in [\lam_{\min}, \lam_{\max}],\\ 
     e^{ -(x - \lam_{\max} )^2 d_{\min} / 2 \eps   },   & \lam_{\max} < x < M.
   \end{cases}
$$
\end{lemma}

\begin{proof}
 For $x \geq \lam_{\max}$ we have, with $\rho$ defined through Lemma~\ref{l:rho}, and $d$ and $\lam$ given in
Definition~\ref{Lspace}  
 \begin{subequations}
   $$ 
     \aligned 
     -\int_\rho^x h(y) \,dy 
     & 
     = \int_{\lam_{\max}}^x d(y) \, \bigl(y - \lam(y)\bigr) \, dy
     +
     \int_{\rho_i}^{\lam_{\max}} d(y) \, \bigl(y - \lam(y)\bigr) \, dy
     \\ 
     & \geq \int_{\lam_{\max}}^x d(y) \, \bigl(y - \lam_{\max}\bigr) \, dy   
     \geq {(x - \lam_{\max})^2 d_{\min} \over 2 }, 
     \endaligned 
   $$
   while a similar argument for $x < \lam_{\min}$ gives  
   $$ 
     -\int_\rho^x h(y) \,dy 
     \geq \int_{\lam_{\min}}^x d(y) \, \bigl(y - \lam_{\min}\bigr) \, dy 
     \geq {(x - \lam_{\min})^2 d_{\min} \over 2 }. 
   $$
 \end{subequations}
 Finally, the desired conclusion follows from definition of $\varphi(\cdot;h)$ and from Lemma~\ref{l:mass}.
\end{proof}


\subsection{Wave coefficients}
 
We rely on the earlier 
work by Joseph and LeFloch~\cite{JosephLeFloch99}, in Lemma~\ref{lemma:linint}. 
Our new contribution is about the case of resonant wave 
coefficients treated in Lemma~\ref{l:quadint}.

\begin{lemma}[Estimates of the wave coefficients]
 \label{lemma:linint}
 The linear coefficients $J_{j \to i}$ defined in \eqref{eq:linear}  satisfy the estimate 
 \be
 \label{interactcoeff}
 \left| J_{j \to i}(y)\right|
 \leq 
 \begin{cases}
   O(\eps) \,( \varphi_i^\star(y) +  \varphi_j^\star(y)),  & i \neq j, 
   \\ 
   2 M \, \varphi_i^\star(y), & i =j,
 \end{cases}
\ee
for all $i,j = 1, \ldots, N$ and $y \in [-M,M]$.
Moreover, the quadratic wave coefficients defined in \eqref{eq:quadra} satisfy  
\be
 \label{interquadra}
 |F_{j,k\to i}(y)| \leq C \, \bigl(\varphi_i^\star(y) + \varphi_j^\star(y) +
 \varphi_k^\star(y) \bigr).
\ee
\end{lemma}

\begin{lemma} 
 \label{l:local}
 By choosing $\delta$ small enough, for all $i\neq j$ there exists
 positive constants $C$ and $D$ independent of $\eps$ such that
 \be
 \left\|\dfrac{\varphi_i^\star}{\varphi_j^\star}\right\|_{L^\infty([\bLam_j,\Lamb_j])}
 \leq C e^{-D/\eps}.
 \ee
\end{lemma}

\begin{lemma}[Resonant quadratic coefficients]
 \label{l:quadint}
Given $\psi\in L^1$ and $i\neq j$, then one has 
 $$
   \left|J_{j\to i}^\psi(y)\right|
   \leq O(1)\,\|\psi\|_1(\varphi^\star_j(y)+\varphi^\star_i(y)),\qquad \eps>0.
 $$
\end{lemma}

\begin{proof}[Proof of Lemma~\ref{lemma:linint}]
The case $i=j$ is obvious, so we only need to consider the case $i \neq j$. 
For definiteness we suppose that $j>i$, the proof for $j<i$ being similar.\\
First, using Lemmas~\ref{l:asymptotic} and \ref{l:mass}, we get in the region $y>c_i$:   
$$
\aligned
\left| \varphi_i^\star(y)\int_{c_i}^y { \varphi_j^\star(x) \over \varphi_i^\star(x)} \,\, dx\right| 
& = \varphi_j^\star(y) \, \int_{c_i}^y \phi(x,y;\mu_i-\mu_j) \,\, dx\\
& \leq \frac{\eps}{\bLam_j-\Lamb_i} \, \varphi_j^\star(y).
\endaligned
$$
On the other hand, in the region $y <c_i$ we have 
$$
\begin{aligned}
 &\left| \varphi_i^\star(y) \int_{c_i}^y { \varphi_j^\star(x)\over \varphi_i^\star(x)} \,\, dx\right| 
  = \varphi_i^\star(y) \, \frac{I_i}{I_j} \, \phi(\rho_j,\rho_i;\mu_j)
 \, \int_y^{c_i}\phi(\rho_i,x,\mu_j-\mu_i) dx\\
  &= \varphi_i^\star(y) \, \frac{I_i}{I_j} \, \phi(\rho_j,\rho_i;\mu_j)
 \, \phi(\rho_i,c_i;\mu_j-\mu_i) \, \int_y^{c_i}\phi(c_i,x,\mu_j-\mu_i) dx.
\end{aligned}
$$
But, by Lemma~\ref{l:asymptotic} we have 
$$
\int_y^{c_i} \phi(c_i,x,\mu_j -\mu_i) dx \leq
{\eps \over \bLam_j - \Lamb_i} 
$$
and, by an easy computation,
$$
\aligned
\phi(\rho_j,\rho_i;\mu_j) &\leq 
e^{- (\bLam_j - \rho_i)^2 / (2 \eps (1+\eta)) },\\
\phi(\rho_i,c_i;\mu_j - \mu_i)
&\leq e^{(\Lamb_j - \bLam_i) \, |c_i - \rho_i|/\eps},\\
I_i/I_j & = O(1/\eps).
\endaligned
$$
Using these observations we also get for all $y < c_i$
$$
\left| \varphi_i^\star(y)\int_{c_i}^y { \varphi_j^\star(x)\over \varphi_i^\star(x)} \,
\, dx\right| \leq C  e^{-\beta_{ij}/\eps}
\varphi_i^\star(y),
$$
with
$$
\aligned
\beta_{ij} 
& = - (\Lamb_j - \bLam_i) \, |c_i - \rho_i| + \frac{(\bLam_j - \rho_i)^2}{2 (1+\eta)}\\
& \geq  - (\Lamb_j - \bLam_i) \, (\Lamb_i - \bLam_i) + \frac{1}{2(1+\eta)}(\bLam_j - \Lamb_i)^2
\endaligned
$$
When $\delta_0$ tends to 0, $(\Lamb_j - \bLam_i)$ remains bounded, $(\Lamb_i - \bLam_i)$ vanishes,
while $(\bLam_j - \Lamb_i)$ tends to $\lam_j(0,0)- \lam_i(0,0) \neq 0$, so assuming $\delta_0$
small enough, we can suppose each quantity $\beta_{ij}$ is positive. The first desired result \eqref{interactcoeff} therefore follows.\\
Using the inequality $|a \, b| \leq (a^2 + b^2)/2$ we have 
$$
|F_{j,k\to i}| \leq {1 \over 2} (|F_{j,j\to i}| + |F_{k,k\to i}|). 
$$
So, we only need to consider the coefficients of the form $F_{jj \to i}$, 
that is, 
$$
G_{i,j}(y): = {\phi(\rho_i,y; \mu_i)\over I_j^2} \left| \int_{c_i}^y 
{\phi(\rho_j,x; \mu_j)^2 \over \phi(\rho_i,x; \mu_i)} \, dx \right|. 
$$
To avoid any distinction between the cases $y>c_i$ and $y<c_i$, the
integral would be noted $\int_{[c_i,y]}$ (this is allowed by the positivity of the integrand). 
Clearly, when $j=i$ we have
$$
G_{i,i} \leq \varphi_i,
$$
since
$$
\int_{[c_i,y]} \phi(\rho_i,x;\mu_i) \, dx \leq I_i.
$$
So, we now suppose $i\neq j$, then
$$
\aligned
G_{i,j}(y) 
& = {\phi(\rho_i,y;\mu_i) \over I_j^2} \int_{[c_i,y]}
\phi(\rho_i,x ; -\mu_i) \, \phi(\rho_j,x; \mu_j)^2 \, dx
\\
& = {\phi(\rho_j,y; \mu_j) \over I_j^2} \int_{[c_i,y]}
\phi(x,y; \mu_i-\mu_j) \, \phi(\rho_j,x ; \mu_j) \, dx
\\
& \leq {\phi(\rho_j,y; \mu_j) \over I_j^2} \int_{[c_i,y]}
\phi(x,y; \mu_i-\mu_j) \, dx,
\endaligned
$$
thanks to the judicious choice of $\rho_j$ that gives
$\phi(\rho_j,x;\mu_j) \leq 1$. Note the important use of this $L^\infty$
estimate that will default in the future. Finally we have
$$
G_{i,j}(y) \leq \frac{1}{I_j} \, J_{j \to i}(y)
$$
which, together with \eqref{interactcoeff} and Lemma~\ref{l:mass} completes this proof . 
\end{proof}

\begin{proof}[Proof of Lemma~\ref{l:local}]
Fix $y\in[\bLam_j,\Lamb_j]$ and recall that
$$
\varphi_i^\star(y) \leq {C \over \eps} \exp\left({{1\over \eps}\int_{\rho_i}^y(\lam_i(t)-t)\,dt}\right).
$$
If $y \geq \rho_i$ then we find 
$$
\aligned
\int_{\rho_i}^y(\lam_i(t)-t)\,dt & \leq \int_{\rho_i}^y
(\Lamb_i-t)\, dt \leq {1 \over 2} \left((\Lamb_i-\rho_i)^2-(\Lamb_i-y)^2\right)\\
& \leq {1 \over 2} \left((\Lamb_i-\bLam_i)^2-(\Lamb_i-y)^2 \right),
\endaligned
$$
while, if $y \leq \rho_i$, 
$$
\aligned
\int_{\rho_i}^y(\lam_i(t)-t)\,dt & =
\int_y^{\rho_i}(t-\lam_i(t))\, dt \leq \int_y^{\rho_i}(t-\bLam_i)\, dt\\
& \leq {1 \over 2} \left((\bLam_i-\rho_i)^2-(\bLam_i-y)^2\right) \leq {1 \over 2} \left((\bLam_i-\Lamb_i)^2-(\bLam_i-y)^2 \right).
\endaligned
$$
If $i<j$ then for all $\eps>0$, $y
\leq \rho_i$ and 
$$
\int_{\rho_i}^y(\lam_i(t)-t)\,dt \leq {1\over 2}\left((\Lamb_i-\bLam_i)^2-(\Lamb_i-\bLam_j)^2\right),
$$
while if $j<i$ then for all $\eps>0$, $y \geq \rho_i$ and 
$$
\int_{\rho_i}^y(\lam_i(t)-t)\,dt \leq {1\over 2}\left((\Lamb_i-\bLam_i)^2-(\bLam_i-\Lamb_j)^2\right),
$$
In all cases, we thus can write
$$
\varphi_i^\star(y)\leq {C \over \eps}\exp\left({{1\over2\eps}\left(\ell_i^2-\Delta_{ij}^2\right)}\right),
$$
where $\ell_i=\Lamb_i-\bLam_i$ represents the width of the i-th wave
and $\Delta_{ij}$ the gap between two different waves $\Delta_{ij}=\min(|\Lamb_i-\bLam_j|,|\Lamb_j-\bLam_i|)$.\\
Moreover, we have
$$
{1\over\varphi_j^\star(y)} \leq 2M
\exp\left({-{1\over\eps}\int_{\rho_j}^y(\lam_j(t)-t)\, dt}\right) \leq 2M e^{ (\Lamb_j-\bLam_j)^2/\eps} \leq 2M e^{\ell_j^2 /\eps}, 
$$
so finally,
$$
{\varphi_i^\star(y)\over\varphi_j^\star(y)} \leq 2MC\eps^{-1}e^{\ell_j^2+\ell_i^2-\Delta_{ij}^2\over\eps}.
$$
Choosing $\delta_0$ small enough, we can assure the positivity of 
all quantities $\Delta_{ij}^2-\ell_j^2+\ell_i^2$, and the result follows.
The multiplicative coefficient $\eps^{-1}$ can be absorbed by reducing the exponential factor.
\end{proof}

\begin{proof}[Proof of Lemma~\ref{l:quadint}]
Using integration by part and denoting $\Psi$ an anti-derivative of $\psi$,
(with so $\|\Psi\|_\infty \leq \|\psi\|_1 < \infty$), we obtain
$$
\aligned
J_{j\to i}^\psi(y)
& =
\varphi^\star_j(y)
\left[\Psi(x){\varphi^\star_j(x)\over\varphi^\star_i(x)}\right]_{c_i}^y-\varphi^\star_i(y)\int_{c_j}^y
\Psi(x){d \over dx}\left({\varphi^\star_j(x)\over\varphi^\star_i(x)}\right)\,dx\\
& =
\varphi^\star_j(y)\Psi(y)-\varphi^\star_i(y)\Psi(c_i){\varphi^\star_j(c_i)\over\varphi^\star_i(c_i)}-\varphi^\star_i(y)\int_{c_i}^y
\Psi(x){d \over dx}\left({\varphi^\star_j(x)\over\varphi^\star_i(x)}\right)\,dx.
\endaligned
$$
The explicit formula for $\varphi^\star_k$ gives
$$
\aligned
{d \over dx}\left({\varphi^\star_j(x)\over\varphi^\star_i(x)}\right) &={d \over dx}\left({I_i\over
   I_j}\exp\Bigl({1\over\eps}\int_{\rho_j}^x(\lam_j(t)-t)\,dt-{1\over\eps}\int_{\rho_i}^x(\lam_j(t)-t)\,dt\Bigr)\right)\\ 
&= 
{1\over\eps}\left(\lam_j(x)-\lam_i(x)\right){\varphi^\star_j(x)\over\varphi^\star_i(x)}
\endaligned
$$
and, consequently,
$$
\aligned
\left|J_{j\to i}^\psi(y)\right|
\leq
\|\Psi\|_\infty\left(\varphi^\star_j(y)+\varphi^\star_i(y)\left\|\dfrac{\varphi^\star_j}{\varphi^\star_i}\right\|_{L\!{}^\infty\!([\bLam_i,\Lamb_i])}\hspace{-3em}+{\|\lam_i-\lam_j\|_\infty\over\eps}\left|\varphi^\star_i(y)\int_{c_i}^y
{\varphi^\star_j(x)\over\varphi^\star_i(x)}\,dx\right|\right).
\endaligned
$$
Thus, by Lemma~\ref{l:local} and Lemma~\ref{lemma:linint} on binary terms for different wave families, 
$$
\left|J_{j\to i}(y)\right| \leq O(\eps)\left(\varphi^\star_j(y)+\varphi^\star_i(y)\right),
$$
we get
$$
\left|J_{j\to i}^\psi(y)\right|\leq
\|\psi\|_1 \left(\varphi^\star_j(y)+\varphi^\star_i(y)C e^{-D/\eps}+O(1)(\varphi^\star_j(y)+\varphi^\star_i(y))\right).
$$
\end{proof}

\begin{remark}
\label{rem:Meth2}
The method just employed could be used to establish Lemma~\ref{lemma:linint} directly,
noting that 
$\|\varphi_k^\star\|_{L^1} =1$ and $J_{j\to i}^{\varphi_k^\star}=F_{j,k\to i}$, and
on the other hand, $\|\varphi_j^\star\|_{L^1}=1$ and
$J_{k\to i}^{\varphi_j^\star}=F_{j,k\to i}$.
We deduce 
$$
\begin{aligned}
\left|F_{j,k\to i}(y)\right|\leq
 O(1)\left(\varphi_j^\star(y)+\varphi_i^\star(y)\right),\\
\left|F_{j,k\to i}(y)\right|\leq
 O(1)\left(\varphi_k^\star(y)+\varphi_i^\star(y)\right), 
\end{aligned}
$$
and \eqref{interquadra} follows.
\end{remark}


\section{Construction of the entropy solution}
\label{construction}

\subsection{Correction vector for a given strength}

Let $C_0(\RR)$ be the space of all 
continuous functions that decay to zero as $|\xi|\to+\infty$,
 and define the following weighted sup-norm for $\theta\in[C_0(\RR)]^N$ as 
$$
\|\theta\|=\sum_{k=1}^N\sup_{\xi\in\RR}{|\theta_k(\xi)|\over\sum_{h=1}^N\varphi^\star_h(\xi)}.
$$
Thus we search for $\theta$ in the Banach space
\be
 E=\Big\{\theta=(\theta_1,\ldots,\theta_N)\in [C_0(\RR)]^N\;:\; \|\theta\|<\infty\Big\}.
\ee 
For $\delta>0$, consider the ball with radius
$\delta$  
$$
 B_\delta := \left\{\tau\in\RR^N\;:\;|\tau|\leq\delta\right\},
$$
and for $\tau\in B_\delta$ 
\be
\label{eq:F}
 \mathcal{F}:=\Big\{\theta\in E\;:\; |\theta_k(\xi)|\leq
 A(\eta|\tau|+|\tau|^2+\nu|\tau|)\sum_h\varphi^\star_h(\xi),\; k=1,\ldots,N\Big\},
\ee
where $A$ is a positive constant to be chosen 
later and $\nu := \|\sigma\|_\infty$. 

The set $\mathcal{F}$ is a closed bounded subset of $E$ in the weighted norm $\|\!\cdot\!\|$. 
The quadratic quantity $|\tau|^2$ is already present in Tzavaras~\cite{Tzavaras96} and comes 
from quadratic coefficients
 between the $\tau_i\varphi^\star_i$ waves. Note however that the presence of the coupling wave $\psi$ (of unit total mass), with its strength $\nu$, enforces the subset $\mathcal{F}$ to contain correction waves that comes from coefficients
 associated with
  $\psi$ and the $\tau_i\varphi^\star_i$, and of strength at most $\nu|\tau|$ relative to the $\varphi^\star_h$.

Now, we will define, for a given strength $\tau$ the correction $\theta(\tau;\cdot)$. Let define the map $T$ that takes $u\in\bar{\Omega}$ where
$$\bar{\Omega}:=\left\{u\in C^0([-M,M]),\|u(\cdot)-u_L\|_{\infty}\leq\varsigma\right\},
\tau\in B_\delta$$
and $\theta\in\mathcal{F}$ to the vector-valued function $T(u,\tau,\theta)$ whose 
components are given by ($k=1,\ldots,N$) 
\be
\label{eq:Ttheta}
\begin{aligned}
&T_k(u,\tau,\theta)(\xi)
\\
& =\eta\ \varphi^\star_k(\xi)\int_{c_k}^\xi{1\over\varphi^\star_k(x)}\sum_i\pi_{ik}(x)\Big(\tau_i\varphi^\star_i(x)+\theta_i(x)\Big)\,dx\\
&\phantom{=}+\varphi^\star_k(\xi)\int_{c_k}^\xi{1\over\varphi^\star_k(x)}\sum_{i,j}\kappa_{ijk}(x)\Big(\tau_i\varphi^\star_i(x)+\theta_i(x)\Big)\Big(\tau_j\varphi^\star_j(x)+\theta_j(x)\Big)\,dx\\
&\phantom{=}+\varphi^\star_k(\xi)\int_{c_k}^\xi{1\over\varphi^\star_k(x)}\sum_i\sigma_{ik}(x)\Big(\tau_i\varphi^\star_i(x)+\theta_i(x)\Big)\psi(x)\,dx.
\end{aligned}
\ee 

\begin{lemma}[Contraction property] 
\label{theta_contr}
 There exists positive constants $A$, $\eta$, $\delta_0$ and $\nu$ such that
 for $\delta<\delta_0$:
 \begin{enumerate}
 \item $T:\bar{\Omega}\times B_\delta\times\mathcal{F}\to\mathcal{F}$ is well defined.
 \item \label{eq:contr_theta} There exists $0<\alpha<1$ such that
   $$
     \big\|T(u,\tau,\theta)-T(u,\tau,\hat{\theta})\big\| \leq \alpha
     \big\|\theta-\hat{\theta}\big\|,\qquad \theta,\hat{\theta}\in \mathcal{F},
   $$
   and for any $u\in\bar{\Omega},\, \tau\in B_\delta$. Therefore
   $T(u,\tau,\cdot):\mathcal{F}\to\mathcal{F}$ is a uniform contraction.
 \item \label{eq:contr_tau} There exists a positive constant $C$, depending on $\mu$ but
   independent of $\delta$, such that
   $$
     \big\|T(u,\tau,\theta)-T(u,\hat{\tau},\theta)\big\| \leq C (\eta+\nu+\delta)
     |\tau-\hat{\tau}|,\qquad \tau,\hat{\tau}\in B_\delta
   $$
   and for any $u\in\bar{\Omega},\,\theta\in\mathcal{F}$.
 \end{enumerate}
\end{lemma}

We deduce from this lemma the following existence result of a correction $\theta(\tau;\cdot)$.

\begin{proposition}
\label{prop_correction}
 Given $u\in\bar{\Omega}, \tau\in B_\delta$, there exists a unique
 $\theta(\tau;\cdot)\in \mathcal{F}$, i.e. in the class of functions
 satisfying
 \be
   \label{eq311}
   |\theta_k(\tau;\cdot)|\leq A(\eta|\tau|+|\tau|^2+\nu|\tau|)\sum_h\varphi^\star_h,\qquad
   |\tau|\leq\delta,\; k=1,\ldots,N,
 \ee
 solution of the fixed point equation $T(u,\tau,\theta)=\theta$.
 Moreover, there exists a constant $C$ independent of $\delta$ such
 that 
 \be
   \label{eq312}
   |\theta_k(\tau;\cdot)-\theta_k(\hat{\tau};\cdot)|\leq
   C(\eta+\nu+\delta)|\tau-\hat{\tau}| \sum_h\varphi^\star_h,\qquad 
   \tau,\hat{\tau}\in B_\delta.
 \ee
\end{proposition}

\begin{proof}[Proof of Lemma~\ref{theta_contr}]
The main difficulty is handling the coupling wave $\nu\psi$.
First of all, we show $T$ keeps the subset $\mathcal{F}$ stable, using the definitions \eqref{eq:F}, \eqref{eq:Ttheta}, and the different definitions of the coefficients \eqref{eq:linear}, \eqref{eq:quadra} and \eqref{eq:coupling}, we get
$$
\begin{aligned}
& \left|T_k(u,\tau,\theta)(\xi)\right|
\\
& \leq \eta\|\pi\|_\infty \sum_i J_{i\to k}(\xi)\Big(|\tau|+A\left(\eta|\tau|+|\tau|^2+\nu|\tau|\right)
\Big)\\
& \phantom{\leq} + \|\kappa\|_\infty\sum_{ij}F_{ij\to k}(\xi)\Big(|\tau|^2+2|\tau|A\left(\eta|\tau|+|\tau|^2+\nu|\tau|\right) 
+A^2\left(\eta|\tau|+|\tau|^2+\nu|\tau|\right)^2 
\Big)\\
& \phantom{\leq} + \|\sigma\|_\infty\sum_i J_{i\to k}^\psi(\xi)\Big(|\tau|+A\left(\eta|\tau|+|\tau|^2+\nu|\tau|\right)
\Big).
\end{aligned}
$$
Thus, by using Lemmas~\ref{lemma:linint} and~\ref{l:quadint} we get
$$
\begin{aligned}
& \left|T_k(u,\tau,\theta)(\xi)\right|\\
& \leq \eta\|\pi\|_\infty\ C_1 N^2\left(|\tau|+A\left(\eta|\tau|+|\tau|^2+\nu|\tau|\right)\right)\sum_h\varphi^\star_h(\xi)\\
& \phantom{\leq} + \|\kappa\|_\infty\ C_2 N^4\Big(|\tau|^2 + 2A|\tau|\left(\eta|\tau|+|\tau|^2+\nu|\tau|\right)
 + A^2\left(\eta|\tau|+|\tau|^2+\nu|\tau|\right)^2\Big)\sum_h\varphi^\star_h(\xi)\\
& \phantom{\leq} + \|\sigma\|_\infty\ C_3 N^2 \left(|\tau|+A\left(\eta|\tau|+|\tau|^2+\nu|\tau|\right)\right)\sum_h\varphi^\star_h(\xi)
\\
& \leq C \left(1+A\left(\eta+\nu+\delta\right)\right)^2(\eta|\tau|+|\tau|^2+\nu|\tau|)\sum_h\varphi^\star_h(\xi),
\end{aligned}
$$
where $C$ is a constant depending only on $N$ the dimension of the
space, on $\|\pi\|_\infty$, $\|\kappa\|_\infty$ and of the constants
$C_1,C_2,C_3$. A necessary condition to get the stability of the subset $\mathcal{F}$
by $T(u,\tau,\cdot)$ is also
$$
C(1+A(\eta+\nu+\delta))^2\leq A
$$
A way to get this inequality is for example, fixing $A=4C$, to choose $\eta$, $\nu$ and
$\delta$ together such that 
$
\eta+\nu+\delta \leq 1/4C.
$

\

Now, 
 $T$ is an uniform contraction relative to the variable $\theta\in\mathcal{F}$, 
 since (from similar arguments) 
$$
\begin{aligned}
& |T_k(u,\tau,\theta)(\xi)-T_k(u,\tau,\hat{\theta})(\xi)| 
\\
& \leq \eta\|\pi\|_{\infty} \sum_j \|\theta-\hat{\theta}\| \sum_i J_{i\to k}(\xi)\\
& \phantom{\leq} + \|\kappa\|_\infty \sum_{ij} \Bigl( 2|\tau|\, \|\theta-\hat{\theta}\| \sum_l F_{il}^k(\xi)
+ 2A\left(\eta|\tau|+|\tau|^2+\nu|\tau|\right) \|\theta-\hat{\theta}\| \sum_{lm}F_{lm\to k}(\xi)\Bigr)\\
& \phantom{\leq} + \|\sigma\|_{\infty} \sum_j \|\theta-\hat{\theta}\| \sum_i J_{i\to k}^\psi(\xi).
\end{aligned}
$$
In view of Lemmas~\ref{lemma:linint} and~\ref{l:quadint} we obtain 
$$
\begin{aligned}
& |T_k(u,\tau,\theta)(\xi)-T_k(u,\tau,\hat{\theta})(\xi)|
\\
& \leq \eta\|\pi\|_{\infty} \ C_1 N^2\|\theta-\hat{\theta}\| \sum_h\varphi^\star_h(\xi)\\
& \phantom{\leq} + \|\kappa\|_\infty\  C_2 N^4 \left(|\tau| + A\left(\eta|\tau|+|\tau|^2+\nu|\tau|\right)\right) \|\theta-\hat{\theta}\| \sum_h\varphi^\star_h(\xi)\\
& \phantom{\leq} + \|\sigma\|_\infty\ C_3 N^2 \|\theta-\hat{\theta}\|\sum_h \varphi^\star_h(\xi)\\
& \leq C(\eta+\nu+\delta)(1+A\delta)\|\theta-\hat{\theta}\|\sum_h\varphi^\star_h(\xi).
\end{aligned}
$$
Finally, we obtain the item (\ref{eq:contr_theta}) of Proposition~\ref{theta_contr} with
$\alpha=C(\eta+\nu+\delta)(1+A\delta)$ by choosing
$\eta+\nu+\delta$ sufficiently small to assure that $\alpha < 1$.

\

Finally,  we have to check the Lipschitz continuity of $T$ in the variable $\tau$.
$$
\begin{aligned}
& |T_k(u,\tau,\theta)(\xi)-T_k(u,\hat{\tau},\theta)(\xi)|\\
& \leq \eta\|\pi\|_\infty |\tau-\hat{\tau}|\sum_i J_{i\to k}(\xi)\\
& \phantom{\leq} + \|\kappa\|_\infty \sum_{ij} \Big(
|\tau_i\hat{\tau}_j-\tau_j\hat{\tau}_i|\, F_{ij\to k}(\xi)
+2|\tau-\hat{\tau}|\,A\delta(\eta+\delta+\nu)\sum_l F_{il\to k}(\xi)\Big)\\
& \phantom{\leq} + \|\sigma\|_\infty |\tau-\hat{\tau}|\sum_i J_{i\to k}^\psi(\xi).
\end{aligned}
$$
By using Lemmas~\ref{lemma:linint} and~\ref{l:quadint} we get
$$
 \begin{aligned}
& |T_k(u,\tau,\theta)(\xi)-T_k(u,\hat{\tau},\theta)(\xi)|
\\
& \leq \eta\|\pi\|_\infty C_1 N |\tau-\hat{\tau}|\sum_h\varphi^\star_h(\xi)
  + \|\kappa\|_\infty C_2 N^2\big( 2\delta |\tau-\hat{\tau}|
  \\
  & \quad +2AN\delta|\tau-\hat{\tau}|\big) \sum_h\varphi^\star_h(\xi)
   + \|\sigma\|_\infty C_3 N |\tau-\hat{\tau}|\sum_h\varphi^\star_h(\xi)\\
& \leq \displaystyle C(\eta+\nu+\delta)|\tau-\hat{\tau}|\sum_h\varphi^\star_h(\xi),
\end{aligned}
$$
and the item (\ref{eq:contr_tau}) of Proposition~\ref{theta_contr} follows. 
\end{proof}

\begin{proof}[Proof of Proposition~\ref{prop_correction}]
 The inequality \eqref{eq311} is a direct consequence of the contraction
 mapping theorem, that previous proposition ensures to apply
 ($\alpha<1$). Let $u\in\bar\Omega$ be given, and $\tau,\hat{\tau}\in B_\delta$, then
 $$
   \begin{aligned}
     \theta(\tau)-\theta(\hat{\tau})
     & = T(u,\tau,\theta(\tau)) - T(u,\hat{\tau},\theta(\hat{\tau}))\\
     & = \big(T(u,\tau,\theta(\tau)) - T(u,\tau,\theta(\hat{\tau}))\big) +
     \big(T(u,\tau,\theta(\hat{\tau})) -  T(u,\hat{\tau},\theta(\hat{\tau}))\big), 
       \end{aligned}
 $$ 
     thus
      $$
   \begin{aligned}
     \|\theta(\tau)-\theta(\hat{\tau})\|
     & \leq \big\|T(u,\tau,\theta(\tau)) - T(u,\tau,\theta(\hat{\tau}))\big\| +
     \big\|T(u,\tau,\theta(\hat{\tau})) - T(u,\hat{\tau},\theta(\hat{\tau}))\big\|\\
     & \leq \alpha \big\|\theta(\tau)-\theta(\hat{\tau})\big\| + C(\eta+\nu+\delta)|\tau-\hat{\tau}|.
   \end{aligned}
 $$
 Hence,
 $$
   \|\theta(\tau)-\theta(\hat{\tau})\|\leq {C\over 1-\alpha}(\eta+\nu+\delta)|\tau-\hat{\tau}|,
 $$
 and \eqref{eq312} ensues.
\end{proof}


\subsection{Strength vector for given Riemann data}

Fix a left-state vector $u_L\in \RN$ and $u\in \bar{\Omega}=\left\{u\in     C^0([-M,M]),\|u(\cdot)-u_L\|_{\infty}\leq\varsigma\right\}$. Being given $\tau\in B_\delta$, we previously constructed a unique $\theta(\tau,\cdot)\in \mathcal{F}$ such that $T(u,\tau,\theta)=\theta$. The question is now to link the vector $\tau\in B_\delta$ to the boundary data $u_L,u_R$. Consider the following operator:
$$ 
S(\tau): =u_L+A_0(u,v)^{-1}\sum_k\int_{-M}^{M}\left[\tau_k\varphi^\star_k(\xi)+\theta_k(\tau,\xi)\right]\hatr_k(u(\xi),v(\xi),\xi)d\xi.
$$

\begin{lemma}
\label{lemma_contraction}
There exist constants $\delta, r>0$ such that the operator 
$P:B_r(u_L)\times\bar{\Omega}\times B_\delta\to B_\delta$ defined by 
$$
 P(u_R,u,\tau)=A_0(u,v)(u_R-u_L)-\sum_k
 \int_{-M}^M\theta_k(\tau,\xi)\hatr_k(u(\xi),v(\xi),\xi)\,d\xi.
$$
satisfies the contraction property (for some $0<\alpha<1$)
$$
\left|P(u_R,u,\tau)-P(u_R,u,\hat{\tau})\right|\leq
\alpha |\tau - \hat{\tau}|,\qquad \tau,\hat{\tau}\in B_\delta
$$
for any $u_R\in B_r(u_L), u\in\bar{\Omega}$. 
\end{lemma}

\begin{proposition}
\label{prop_strenght}
Given $u\in \bar{\Omega}$, there exist positive constants $r$ and $\delta$ such that the following holds: 
\begin{enumerate}

\item For all $u_R\in B_r(u_L)$ there exists a unique solution of the equation $S(\tau)=u_R$ with $\tau\in B_\delta$.

\item For each $u\in \bar{\Omega}$ and $\eps>0$ the inverse map $S^{-1}:B_r(u_L)\to B_\delta$ is well defined and satisfies
\be
\label{eq:inversemap}
|S^{-1}(u_R)|\leq \gamma|u_R-u_L|,
\ee
where $\gamma$ is a constant depending on $\varsigma$ but independent of $u\in \bar{\Omega}$ and $\eps$.
\end{enumerate}
\end{proposition}

\begin{proof}[Proof of Lemma~\ref{lemma_contraction}]
Letting $u_R\in B_r(u_L)$, $u\in\bar{\Omega}$ and $\tau\in B_\delta$, one has 
$$
\begin{aligned}
|P(u_R,u,\cdot)|&\leq \|A_0\||u_R-u_L|+ N^2 R
A(\eta|\tau|+|\tau|^2+\nu|\tau|)\\
&\leq \left( \|A_0\|r + R N^2 A (\eta\delta+\delta^2 +\nu\delta)\right). 
\end{aligned}
$$
Hence, one has the inclusion $P(u_R,u,B_\delta)\subset B_\delta$ provided 
$$
\left( \|A_0\|r + R N^2 A (\eta\delta+\delta^2 +\nu\delta)\right)\leq \delta, 
$$
that is by choosing $r$, $\eta$, $\delta$ and $\nu$ such that
$$
R N^2 A (\eta+\delta+\nu)\delta \leq \delta /2,
$$
$$
\|A_0\|\beta r \leq \delta /2, 
$$
that is to say
$$
\eta+\delta+\nu \leq 1/2R N^2 A,
$$
$$
r \leq \delta / 2 \|A_0\|.
$$
Given $\tau$ and $\hat{\tau}$ in $B_\delta$ we have 
$$
P(u_R,u,\tau)-P(u_R,u,\hat{\tau})=\sum_k \int_{-M}^M\!\!\!\!\!\!\left[\theta_k(\xi,\tau)-\theta_k(\xi,\hat{\tau})\right]\hatr_k(u(\xi),v(\xi),\xi)\,d\xi,
$$
$$
\begin{aligned}
\left|P(u_R,u,\tau)-P(u_R,u,\hat{\tau})\right|
& \leq \sum_k
\int_{-M}^{M}\left|\theta_k(\xi,\tau)-\theta_k(\xi,\hat{\tau})\right|\,|\hatr_k(u(\xi),v(\xi),\xi)|\,d\xi\\
& \leq R N C (\eta+\delta+\nu)|\tau-\hat{\tau}|\sum_k
\int_{-M}^{M}\varphi^\star_k(\xi)\,d\xi\\
& \leq N^2 C (\eta+\delta+\nu)|\tau-\hat{\tau}|.
\end{aligned}
$$
Provided 
\be
\alpha := N^2 C (\eta+\delta+\nu) < 1,
\ee
the map $P(u_R,u,\cdot)$ is a uniform contraction on $B_\delta$.
\end{proof}

\begin{proof}[Proof of Proposition~\ref{prop_strenght}]
Let $u_L$ be fixed. The equation $S(\tau)=u_R$ takes the form
$$
\begin{aligned}
&A_0(u,v)(u_R-u_L)\\
&= \sum_k\tau_k \int_{-M}^{M}\varphi^\star_k(\xi)\hatr_k(u(\xi),v(\xi),\xi)d\xi + \sum_k \int_{-M}^{M}\theta_k(\tau,\xi)\hatr_k(u(\xi),v(\xi),\xi)d\xi,
\end{aligned}
$$
in other words $\tau$ solves the equation
\be
\label{eq:tau}
A_0(u,v)(u_R-u_L)=C(u,v)\ \tau +  \sum_k \int_{-M}^{M}\theta_k(\tau,\xi)\hatr_k(u(\xi),v(\xi),\xi)d\xi,
\ee
where $C(u,v)$ is the matrix whose $k$-th columns is given by
$$
\int_{-M}^{M}\varphi^\star_k(\xi)\hatr_k(u(\xi),v(\xi),\xi)d\xi,\qquad k=1,\ldots,N.
$$
This matrix has the important property it is invertible for any $u\in\bar{\Omega}$ and the inverse matrix $C(u,v)^{-1}$ is uniformly bounded (cf~\cite{Tzavaras96})
\be
|C(u,v)^{-1}|\leq \beta, \qquad u\in\bar{\Omega}.
\ee
In order to solve the equation $S(\tau)=u_R$, observe that solutions of \eqref{eq:tau} are also fixed points of the map $\tau\mapsto C(u,v)^{-1}P(u_R,u,\tau)$, whose existence are ensured by Lemma~\ref{lemma_contraction}.
As a consequence, given $u_R\in B_r(u_L)$, there exists a unique fixed point $\tau$ of
$P(u_R,u,\cdot)$ in the ball $B_\delta$. Moreover it also satisfies
$$
\begin{aligned}
|\tau| &\leq
\left|A_0(u,v)C(u,v)^{-1}\right||u_R-u_L|
\\
& \quad +\left|C(u,v)^{-1}\right|\sum_k \int_{-M}^{M}\!\!\!\!\!\!\!|\theta_k(\xi,\tau)|\,|\hatr_k(u(\xi),v(\xi),\xi)|d\xi, 
\end{aligned}
$$
thus 
$$
\begin{aligned}
|\tau| 
& \leq \|A_0\|\beta |u_R-u_L|  + \beta R A N^2 (\eta|\tau|+|\tau|^2+\nu|\tau|)\\
&\leq  \|A_0\|\beta |u_R-u_L| + 1/2 |\tau|.
\end{aligned}
$$
Thus, $|\tau|\leq 2 \|A_0\|\beta |u_R-u_L|$, which finally implies \eqref{eq:inversemap}.
\end{proof}


\subsection{Riemann problem}

We search for a solution of \eqref{2.12b} under the form \eqref{2.12}
satisfying the boundary conditions \eqref{2.10} and where $v$ is known by \eqref{2.11}.
 
\begin{theorem}[Uniform estimates and existence result] 
\label{thm_TV} 
There exists a solution $u^\eps\in\bar\Omega$  
of the problem \eqref{2.10}-\eqref{2.12b} satisfying, for some constant $K$ independent of $\eps$, 
\be
\begin{aligned}
&TV(u^\eps) \leq K |u_R-u_L|,\\
&\eps |u^\eps_\xi| \leq K.
\end{aligned}
\ee
After extracting a subsequence if necessary, this result provides us with a solution with bounded total variation. 
\end{theorem}

\begin{proof}
$\varsigma>0$ is choosen so that conditions of eigenvalue separation are
fulfilled on $\bar{\Omega}$. Fix $u_L$ and $u\in {\bar\Omega}$. For $\eps$ fixed, we construct $z$ as
$$
z(\xi)=u_L+A_0(u,v)^{-1}\int_{-M}^\xi\sum_{j=1}^n
\left(\tau_j\varphi_j^\star(\zeta)+\theta_j(\tau;\zeta)\right)\hatr_j(u(\zeta),v(\zeta),\zeta)\,d\zeta
$$
by following steps:
\begin{enumerate}
\item Each $\varphi_j^\star$ is constructed as the fundamental wave measure from \eqref{2.15}, recalling 
$$
\varphi_j^{\star}{}' - \frac{\mu_j(u,v, \cdot)}{\eps} \, \varphi_j^\star=0.
$$

\item For each $\tau$ small enough we can get, through Proposition~\ref{prop_correction}, a correction $\theta(\tau,\cdot)$ so that $a_j=\tau_j\varphi_j^\star+\theta_j$ is solution of \eqref{2.13a} 
$$
   a_j' - \frac{\mu_j(u,v, \cdot)}{\eps} \, a_j =  \eta  L_j(u,v,\cdot)\, + 
   Q_j(u,v, \cdot) \, + S_j(u,v, \cdot).
$$
\item The vector of strength $\tau$ is then chosen, through Proposition~\ref{prop_strenght}, as a solution of $S(\tau)=u_R\in B_r(u_L)$. This way the solution $\ut$ of
$$
A_0(u,v)\,\ut_\xi(\xi) = \sum_j \, a_j(\xi) \, \hatr_j(u, v, \xi)
$$
satisfying $\ut(-M)=u_L$, satisfies moreover $\ut(M)=u_R$.
\end{enumerate}
These steps allow us to construct an operator $\mathcal{T}:{\bar\Omega}\rightarrow E, u\mapsto z$, and $\mathcal{T}(u)=z\in{\bar\Omega}$. We only need to get a fixed point result on $\mathcal{T}$ to get the solution $u$ of the whole problem, and then sufficiently strong estimates to ensures existence of the limit as $\eps$ tends to 0.
\end{proof}

\begin{lemma}
\label{l:thickness}
The function $v^\eps$ converges toward the sign function 
(denoted by $\sgn$) and, more precisely, for all $c>0$
\be
\left\|v^\eps-\sgn\right\|_{L^\infty(\RR\setminus[-c,c])}=o(\eps).
\ee
\end{lemma}

\begin{proof}
Indeed the formula \eqref{2.11} implies $v^\eps$ takes the form
$$
 v^\eps(\xi)=-1+2
 {\int_{-M}^\xi e^{-{x^2\over2 \eps^p} dx} \over \int_{- M}^{
     M}e^{-{x^2\over 2\eps^p}} dx}.
$$
Fix $\xi>c>0$, so that 
$$
\aligned 
\left|v^\eps(x)-1\right| 
& \leq 2{\int_{c}^M e^{-{x^2\over2 \eps^p} dx} \over \int_{- M}^{
     M}e^{-{x^2\over 2\eps^p}} dx}
     \\
     & \leq 2 {M e^{-{c^2\over2 \eps^p}}\over \eps^{p/2} \int_{- M/\eps^{p/2}}^{
     M/\eps^{p/2}}e^{-{y^2\over 2}} dy}\leq {C \over \eps^{p/2}} e^{-{c^2\over2 \eps^p}}.
\endaligned
$$
For $\xi<-c<0$, by the same procedure, we get
$$
\left|v^\eps(x)+1\right| \leq {C \over \eps^{p/2}} e^{-{c^2\over2 \eps^p}},
$$
the lemma is therefore proved.
\end{proof}

\begin{theorem}[Convergence to an entropy solution]
\label{thm-convergence-dafermos}
The sequence $u^\eps$ converges pointwise toward $u\in BV$, satisfying 
\be
\begin{aligned}
-\xi\ \dfrac{d}{d\xi}\lpcL(u) + \dfrac{d}{d\xi} f_-(\lpcL(u))=0, \qquad \text{ in } \xi <0, 
\\
-\xi\ \dfrac{d}{d\xi}\lpcR(u) + \dfrac{d}{d\xi} f_+(\lpcR(u))=0, \qquad \text{ in } \xi >0.
\end{aligned}
\ee
Let $\eta_{\pm}=\eta_{\pm}(u)\in\RN$ be two entropy functions compatible with the viscosity matrix in the sense that
$$
\nabla^2\eta_\pm(u) B_{0\pm}(u)\geq 0,\qquad u\in\calU.
$$
Then following entropy inequalities are satisfied, 
\be
\begin{aligned}
-\xi\ \dfrac{d}{d\xi}\eta_-(\lpcL(u)) + \dfrac{d}{d\xi} q_-(\lpcL(u))\leq 0,\qquad \text{ in } \xi <0, 
\\
-\xi\ \dfrac{d}{d\xi}\eta_+(\lpcR(u)) + \dfrac{d}{d\xi} q_+(\lpcR(u))\leq 0, \qquad \text{ in } \xi >0. 
\end{aligned}
\ee
\end{theorem}

\begin{proof}
Let $\calO$ be a neighborhood of 0, the solution $u^\eps$ of \eqref{1.3-NEW} is obtained as a smooth function so that, under the consistancy hypothesis \eqref{Consistancy}, we have outside $\calO$
$$
\begin{aligned}
A_0(u^\eps,v^0){d u^\eps\over d\xi} &= \dfrac{d}{d\xi}\lpcLR(u^\eps),\\
A_1(u^\eps,v^0){d u^\eps\over d\xi} &= \dfrac{d}{d\xi} f_\pm(\lpcLR(u^\eps)).
\end{aligned}
$$
Let $\phi\in C^\infty_c(\RR_-\setminus\calO)$ be a test-function with a compact support included in $\RR_-\setminus\calO$, then \eqref{1.3-NEW} implies
$$
\begin{aligned}
&\int_\RR \bigl(-\xi\ {d \over d\xi} \lpcL(u^\eps) + \dfrac{d}{d\xi} f_-(\lpcL(u^\eps))\bigr)\phi\;d\xi\\
& = \int_\RR -\xi\ \left(A_0(u^\eps,v^0)-A_0(u^\eps,v^\eps)\right) u_\xi^\eps\phi\;d\xi\\
& \quad+ \int_\RR \left(A_1(u^\eps,v^0)-A_1(u^\eps,v^\eps)\right) u_\xi^\eps\phi\;d\xi
 + \int_\RR \eps\left(B_0(u^\eps,v^\eps) u_\xi^\eps\right)_\xi \phi \;d\xi.
\end{aligned}
$$
Moreover, thanks to Lemma~\ref{l:thickness} and using Lipschitz continuity properties, we have
$$
\begin{aligned}
\left| \int_\RR -\xi\left(A_0(u^\eps,v^0)-A_0(u^\eps,v^\eps)\right) u_\xi^\eps\phi\;d\xi\right|
&\leq o(\eps) \textrm{Lip}(A_0) \|\xi \phi\|_\infty TV(u^\eps),
\\
\left| \int_\RR \left(A_1(u^\eps,v^0)-A_1(u^\eps,v^\eps)\right)u_\xi^\eps\phi\;d\xi\right|
&\leq o(\eps) \textrm{Lip}(A_1) \|\phi\|_\infty TV(u^\eps),
\end{aligned}
$$
and 
$$
\aligned
\eps \left| \int_\RR \left(B_0(u^\eps,v^\eps) u_\xi^\eps\right)_\xi \phi\;d\xi\right|
& = \eps \left| \int_\RR B_0(u^\eps,v^\eps) u_\xi^\eps \phi_\xi\;d\xi\right| 
\\
& \leq  \eps \|B_0\|_\infty \|\phi_\xi\|_\infty TV(u^\eps).
\endaligned
$$
Thus, as $\eps$ tends to 0, we get the weak formulation for the limit $u$
$$
\int_\RR \bigl(-\xi\ \dfrac{d}{d\xi}\lpcL(u) + \dfrac{d}{d\xi} f_-(\lpcL(u))\bigr)\varphi\;d\xi  = 0.
$$
By a similar method we get for $\varphi\in C^\infty_c(\RR_+\setminus\calO)$
$$
\int_\RR \bigl(-\xi\ \dfrac{d}{d\xi}\lpcR(u) + \dfrac{d}{d\xi} f_+(\lpcR(u))\bigr)\varphi\;d\xi  = 0.
$$
Entropy inequalities are obtained by first using the consistency hypothesis \eqref{Consistancy}, that give outside $\calO$
$$
\begin{aligned}
\nabla\eta_\pm(u^\eps)\cdot A_0(u^\eps,v^0)u_\xi^\eps &= \dfrac{d}{d\xi} \eta_{\pm}(\lpcLR(u^\eps)),\\
\nabla\eta_\pm(u^\eps)\cdot A_1(u^\eps,v^0)u_\xi^\eps &= \dfrac{d}{d\xi} q_\pm(\lpcLR(u^\eps)).
\end{aligned}
$$
Let $\phi\in C^\infty_c(\RR_-\setminus\calO)$ be a non-negative test function with a compact support included in $\RR_-\setminus\calO$, then \eqref{1.3-NEW} implies
$$
\begin{aligned}
&\int_\RR \bigl(-\xi\ \dfrac{d}{d\xi}\eta_-(\lpcL(u^\eps)) + \dfrac{d}{d\xi} q_-(\lpcL(u^\eps))\bigr)\phi\;d\xi\\
\qquad\qquad& = \int_\RR -\xi\ \nabla\eta_-(u^\eps)\cdot \left(A_0(u^\eps,v^0)-A_0(u^\eps,v^\eps)\right) u_\xi^\eps\phi\;d\xi
\\
& \quad + \int_\RR \nabla\eta_-(u^\eps)\cdot \left(A_1(u^\eps,v^0)-A_1(u^\eps,v^\eps)\right) u_\xi^\eps\phi\;d\xi
\\
& \quad + \int_\RR \eps\nabla\eta_-(u^\eps)\cdot \left(B_0(u^\eps,v^\eps) u_\xi^\eps\right)_\xi \phi \;d\xi.
\end{aligned}
$$
With similar arguments as previously, the first and the second terms of right hand side tends to 0 as $\eps$ tends to 0. Moreover, after reporting the $\xi$-derivative on $\phi\nabla\eta_-(u^\eps)$, the last term 
$$
\begin{aligned}
& - \int_\RR \eps\phi_\xi\nabla\eta_-(u^\eps)\cdot B_0(u^\eps,v^\eps) u_\xi^\eps \;d\xi
  - \int_\RR \eps\phi\nabla^2\eta_-(u^\eps)\cdot B_0(u^\eps,v^0) |u^\eps_\xi|^2 \;d\xi\\
&+ \int_\RR \eps\phi\nabla^2\eta_-(u^\eps)\cdot (B_0(u^\eps,v^0)-B_0(u^\eps,v^\eps)) |u^\eps_\xi|^2  \;d\xi
\end{aligned}
$$
satisfies the estimates
$$
\begin{aligned}
&\Big| 
\int_\RR \eps \phi_\xi\nabla\eta_-(u^\eps)\cdot B_0(u^\eps,v^\eps) u_\xi^\eps \;d\xi\Big|
\leq K\eps TV(u^\eps) 
\end{aligned}
$$
and
$$
\begin{aligned}
&\left|\int_\RR \eps \phi\nabla^2\eta_-(u^\eps)\cdot (B_0(u^\eps,v^0)-B_0(u^\eps,v^\eps)) |u^\eps_\xi|^2 \;d\xi\right|\\
&\leq K \textrm{Lip}(B_0)\|\nabla^2\eta_-\|_\infty\|\phi\|_\infty TV(u^\eps)\|v^\eps-v^0\|_{L^\infty(\RR\setminus\calO)}.
\end{aligned}
$$
However the quantity $\int_\RR \eps\phi\nabla^2\eta_-(u^\eps)\cdot B_0(u^\eps,v^0) |u^\eps_\xi|^2  \;d\xi$ is not 
guaranteed
 to vanish as $\eps$ tends to 0, but it converges toward a positive value under the hypothesis $\nabla^2\eta_\pm B_0{}_\pm(u)\geq 0$. The following weak formulation of the entropy inequality on $\mathbb{R}_-$ follows: 
$$
\int_\RR \bigl(-\xi\ \dfrac{d}{d\xi}\eta_-(\lpcL(u)) + \dfrac{d}{d\xi} q_-(\lpcL(u))\bigr)\phi\;d\xi\leq 0. 
$$
Similar arguments lead to the entropy inequality on $\RR_+$.
\end{proof}


\section*{Acknowledgments}
The authors were partially supported by the Centre National de la Recherche
Scientifique (CNRS),
the Commissariat \`a l'\'Energie Atomique (Saclay), 
and the Agence Nationale de la Recherche (ANR) through Grant 06-2-134423.


\newcommand{\auth}{\textsc}

\end{document}